\documentclass[reqno, 11pt]{amsart}
\usepackage{amssymb}
\usepackage{color}
\usepackage{verbatim}
\newcommand{\mysection}[1]{\section{#1}
\setcounter{equation}{0}}

\newtheorem{theorem}{Theorem}[section]
\newtheorem{corollary}[theorem]{Corollary}
\newtheorem{lemma}[theorem]{Lemma}
\newtheorem{proposition}[theorem]{Proposition}

\theoremstyle{definition}
\newtheorem{remark}[theorem]{Remark}

\theoremstyle{definition}

\theoremstyle{definition}
\newtheorem{assumption}[theorem]{Assumption}

\makeatletter
\def\dashint{\operatorname%
{\,\,\text{\bf--}\kern-.98em\DOTSI\intop\ilimits@\!\!}}
\makeatother

\newcommand{\WO}{\mathring{W}}

\newcommand\bL{\mathbb{L}}

\def\bR{\mathbb{R}}
\def\bZ{\mathbb{Z}}

\def\bC{\mathbb{C}}
\def\bB{\mathbf{B}}

\def\cC{\mathcal{C}}

\def\cM{\mathcal{M}}
\def\cO{\mathcal{O}}

\def\cL{\mathcal{L}}

\newcommand{\Div}{\operatorname{div}}

\title [Elliptic equations with irregular coefficients]{On elliptic equations in a half space or in convex wedges with irregular coefficients}

\author[H. Dong]{Hongjie Dong}
\address[H. Dong]{Division of Applied Mathematics, Brown University, 182 George Street, Box F, Providence, RI 02912, USA}
\email{Hongjie\_Dong@brown.edu}
\thanks{Email: Hongjie\_Dong@brown.edu, Tel: 1-4018637297, Fax: 1-4018631355. The author was partially supported by the NSF under agreement DMS-0800129 and DMS-1056737.}

\date{\today}



\begin{document}

\begin{abstract}
We consider second-order elliptic equations in a half space with
leading coefficients measurable in a tangential direction. We
prove the $W^2_p$-estimate and solvability for the Dirichlet
problem when $p\in (1,2]$, and for the Neumann problem when $p\in
[2,\infty)$. We then extend these results to equations with more
general coefficients, which are measurable in a tangential
direction and have small mean oscillations in the other
directions. As an application, we obtain the $W^2_p$-solvability
of elliptic equations in convex wedge domains or in convex polygonal domains with discontinuous coefficients.
\end{abstract}

\maketitle

\mysection{Introduction}
                                    \label{sec1}

In this paper we study second-order elliptic equations in non-divergence form:
\begin{equation*}
               %                                 \label{parabolic}
L u-\lambda u=f
\end{equation*}
in a half space,
where $\lambda\ge 0$ is a constant and
\begin{equation*}
              %                                      \label{9.30.3}
L u=a^{ij}D_{ij}u+b^{i}D_i u+cu
\end{equation*}
is an uniformly elliptic operator with bounded and measurable coefficients. The leading coefficients $a^{ij}$ are symmetric, merely measurable in a tangential direction, and either independent or have very mild regularity in the orthogonal directions. This type of equations typically arises in homogenization of layered materials with boundaries perpendicular to the layers.

The $L_p$ theory of non-divergence form second-order  elliptic and
parabolic equations with discontinuous coefficients was studied
extensively by many authors. According to the well-known
counterexample of Nadirashvili, in general there does not exist
solvability theory for uniformly elliptic operators with general
bounded and measurable coefficients. In the last fifty years, many
efforts were made to treat particular types of discontinuous
coefficients. The $W^2_2$-estimate for elliptic equations with
measurable coefficients in smooth domains was obtained by Bers and
Nirenberg \cite{BN54} in the two dimensional case in 1954, and by
Talenti \cite{Ta65} in any dimensions under the Cordes condition.
In \cite{Ca67} Campanato established the $W^2_p$-estimate for
elliptic equations with measurable coefficients in 2D for $p$ in a
neighborhood of $2$. A corresponding result for parabolic
equations can be found in Krylov \cite{Kr70}. By using explicit
representation formulae, Lorenzi \cite{Lo72a, Lo72b} studied the
$W^2_2$ and $W^2_p$ estimates for elliptic equations with
piecewise constant coefficients in the upper and lower half
spaces. See also \cite{Sa76} for a similar result for parabolic
equations and a recent paper \cite{Ki07} by Kim for elliptic
equations in $\bR^d$ with leading coefficients which are
discontinuous at finitely many parallel hyperplanes. In \cite{Chi}
Chiti obtained the $W^2_2$-estimate for elliptic equations  in
$\bR^d$ with coefficients which are measurable functions of a
fixed direction.

Another notable type of discontinuous coefficients contains functions with vanishing
mean oscillation (VMO) introduce by Sarason. The study of elliptic and parabolic equations with VMO coefficients was initiated by Chiarenza, Frasca, and Longo \cite{CFL1} in 1991 and continued in \cite{CFL2} and \cite{BC93}.
We also refer the read to Lieberman \cite{Li03} for an elementary treatment of elliptic equations with VMO coefficients in Morrey spaces.
%The proofs in \cite{CFL1, CFL2, BC93} are based on the Calder\'on-Zygmund theorem and the Coifman-Rochberg-Weiss commutator theorem.
In \cite{Krylov_2005} Krylov gave a unified
approach to investigating the $L_p$ solvability of both
divergence and non-divergence form  elliptic and parabolic equations in the whole space with
leading coefficients that are in VMO in the spatial
variables (and measurable in the time variable in the parabolic case).
Unlike the arguments in \cite{CFL1,CFL2,BC93} which are based on the certain representation formulae and Calder\'on--Zygmund theory, the proofs in
\cite{Krylov_2005} rely mainly on pointwise estimates of sharp
functions of spatial derivatives of solutions, so that
VMO coefficients are treated in a rather straightforward
manner. Later this approach was further developed to treat more general types of coefficients. In \cite{KimKrylov07} Kim and Krylov established the $W^2_p$-estimate, for $p\in (2,\infty)$, of elliptic equations in $\bR^d$ with leading coefficients measurable in a fixed direction and VMO in the orthogonal directions, which, in particular, generalized the result in \cite{Chi}. By using a standard method of odd and even extensions, their result carries over to equations in a half space when the leading coefficients are measurable in the normal direction and VMO in all the tangential directions. Recently, the results in \cite{KimKrylov07} were extended in \cite{Krylov08} and \cite{D_TAMS11}, in the latter of which the restriction $p>2$ was dropped. We also mention that in \cite{DK_TAMS10} the $W^2_p$-estimates, for $p\ge 2$ close to 2, were obtained for elliptic and parabolic equations. The leading coefficients are assumed to be measurable in the time variable and {\em two} coordinates of space variables, and almost VMO with respect to the other coordinates. In particular, these results extended the aforementioned results in \cite{Ca67} and \cite{Kr70} to high dimensions.

Since the work in \cite{KimKrylov07}, the following problem remains open: {\em Do we have a $W^2_p$-estimate for uniformly elliptic operators in a half space with leading coefficients measurable in a tangential direction and, say, with the homogeneous Dirichlet boundary condition?}

The main objective of the paper is to fully answer this problem.
Apparently, in this case the estimate does not follow from the results in \cite{KimKrylov07} and \cite{D_TAMS11} by using the method of odd and even extensions, since after even and odd extensions one obtains an elliptic equation in the whole space with leading coefficients measurable in two directions, i.e., a tangential direction and the normal direction. In fact, the answer to the problem is negative for general $p>2$: for any $p>2$ there is an elliptic operator with ellipticity constant depending on $p$, such that the $W^2_p$ estimate does not hold for this operator; cf. Remark \ref{rm5.46}.
Nevertheless, by invoking Theorems 2.2
and 2.8 of \cite{DK_TAMS10}, the extension argument does give an affirmative answer to the problem for any $p\in [2,2+\varepsilon)$ close to 2, where $\varepsilon$ depends on the dimension and the ellipticity constant. In this paper, we shall give an affirmative answer in the remaining case $p\in (1,2)$. We also consider equations with the Neumann boundary condition, in which case we prove the $W^2_p$ estimate for any $p\in [2,\infty)$.

To precisely describe our results (Theorems \ref{thm1} and \ref{thm2}), we first introduce a few notation. Let $d\ge 2$ be an integer and $\lambda\ge 0$ be a constant.
A typical point
in $\bR^{d}$ is denoted by $x=(x^{1},...,x^{d})=(x',x'')$, where $x'=(x^1,x^2)\in \bR^2$ and $x''=(x^3,\ldots,x^d)\in \bR^{d-2}$.  Consider the elliptic equation
\begin{equation}
                                \label{eq8.51}
Lu-\lambda u:=a^{ij}(x^2)D_{ij}u-\lambda u=f
\end{equation}
in a half space $\bR^d_+=\{x\in \bR^d\,|\,x^1>0\}$ with the homogeneous Dirichlet or Neumann boundary condition.
Assume that $a^{ij}$ are measurable, bounded, symmetric, and uniformly elliptic, i.e., there is a constant $\delta\in (0,1)$ such that for any unit vector $\xi\in \bR^d$, we have
\begin{equation}
                                        \label{eq17.53}
\delta\le a^{ij}\xi^i\xi^j,\quad |a^{ij}|\le \delta^{-1}.
\end{equation}
The main results of the paper are as follows. Suppose $p\in (1,2]$ (or $p\in [2,\infty)$) and $u\in W^2_p(\bR^d_+)$ satisfies \eqref{eq8.51} with the Dirichlet boundary $u=0$ (or the Neumann condition $D_1 u=0$, respectively) on $\partial\bR^d_+$. Then the following apriori estimate holds:
$$
\lambda\|u\|_{L_{p}(\bR^d_+)}+\lambda^{1/2}
\|Du\|_{L_{p}(\bR^d_+)}+\|D^2u\|_{L_{p}(\bR^d_+)}\le
N(d,\delta,p)\|f\|_{L_{p}(\bR^d_+)}.
$$
Moreover, for any $\lambda>0$ and $f\in L_p(\bR^d_+)$, there is a unique solution $u\in W^2_p(\bR^d_+)$ to the problem. The range of $p$ is sharp in the Dirichlet case; see Remark \ref{rm5.46}.

Let us give a brief description of the proofs. We note that since
the coefficients are merely measurable functions, the classical
Calder\'on--Zygmund approach no longer applies. Furthermore,
solutions to the homogeneous problem generally only possess
$C^{1,\alpha}$ regularity near the boundary. So the approach in
\cite{Krylov_2005} cannot be applied directly either. Our main
idea of the proofs is that, thanks to the assumption
$a^{ij}=a^{ij}(x^2)$, after a change of variables we can rewrite
$L$ into a divergence form operator of certain type (see
\eqref{eq10.41}). In the case of the Dirichlet boundary condition,
we shall show that $v:=D_1 u$ satisfies a divergence form equation
with the conormal derivative boundary condition. This crucial
observation was used before by Jensen \cite{Je80} and Lieberman
\cite{Li92} in different contexts. While in the case of the
Neumann boundary condition, we show that $v$ satisfies the same
divergence form equation with the Dirichlet boundary condition. Then the problem
is reduced to study the $W^1_p$-estimate for these divergence form
operators. By a duality argument, it suffices to focus on the case
$p>2$. However, due to the lack of regularity of the coefficients
in the $x^2$ direction, the usual interior and boundary
$C^{\alpha}$ estimates of $Dv$ do not hold. Nevertheless, we apply
the DeGiorgi--Nash--Moser estimate, the reverse H\"older's
inequality, and an anisotropic Sobolev inequality to get boundary
and interior $C^\alpha$ estimate of certain linear combination of
$Dv$ (Lemmas \ref{lem2.1} and \ref{lem2.2}), using also the
properties of the operators. With this $C^\alpha$ estimate at
disposal, we are able to use Krylov's approach mentioned above to
established the desired $W^1_p$-estimate of $v$.

In the paper, we also treat elliptic equations with more general coefficients which are measurable in a tangential direction and VMO with respect to the other variables. We obtain similar $W^2_p$-estimates for both the Dirichlet problem and the Neumann problem (cf. Theorems \ref{thm3} and \ref{thm4}), generalizing Theorems \ref{thm1} and \ref{thm2}.

The significance of these results is that they can be used to deduce new $W^2_p$-estimate for elliptic equations with discontinuous coefficients in convex wedges or convex polygonal domains. More precisely, we obtain the $W^2_p$-estimate, $p\in (1,2+\varepsilon)$, for non-divergence form elliptic equation with VMO coefficients in a convex wedge $\Omega_\theta=\cO_\theta\times \bR^{d-2}$, where
$$
\theta\in (0,\pi),\quad
\cO_\theta=\{x'\in\bR^2\,|\,x_1> |x'|\cos(\theta/2)\},
$$
with the homogeneous Dirichlet boundary condition; cf. Theorem \ref{thm5.1} and the remark below it. The range of $p$ is sharp even for Laplace equations or equations with constant coefficients. See Remark \ref{rm5.46}, \cite[Theorem 4.3.2.4]{Gr85}, or \cite[Sect. 4.3.1]{MR08}.

There is a vast literature on the $L_p$ theory for elliptic and parabolic equations in domains with wedges or with conical or angular points. See, for instance, \cite{Ko67,Gr85, KO83,Da88,NP94, KMR97,MR08} and the references therein. For Laplace equations in convex domains or in Lipschitz domains, we also refer the reader to \cite{Ad93,Fr93,JK95,FMM98}. It is worth mentioning that Hieber and Wood \cite{HW07} extended the aforementioned results in \cite{Ta65, Ca67} to equations with measurable coefficients in bounded convex domains. In \cite{Lo75,Lo77a} Lorenzi considered elliptic equations with piecewise constant coefficients in two sub-angles of an angular domain in the plane with zero right-hand side and inhomogeneous Dirichlet boundary conditions. Heat equations in domains with wedges were studied in \cite{So84, So01, Na01}. However, in all these references, either the leading coefficients are assumed to be constants or sufficiently regular, or $p$ needs to be in a small neighborhood of $2$. To our best knowledge, Theorem \ref{thm5.1} appears to be the first of its kind about non-divergence elliptic equations in convex wedge domains with discontinuous coefficients, which does not impose any additional condition on $p$.

The paper is organized as follows. In the next section we consider divergence form elliptic operators of two different types. We obtain $W^1_p$-estimate for these operators, which are crucial in the proofs of the main results in Section \ref{sec3}. In Section \ref{sec4}, we treat coefficients which are measurable in $x^2$ and VMO with respect to other variables. We discuss the applications to equations in convex wedges or convex polygonal domains in Section \ref{sec5}.

\mysection{Some auxiliary estimates}
                    \label{sec2}

We introduce $\bL_1$ as the collection of divergence form operators
$$
\cL u=D_i(a^{ij}D_{j}u)
$$
satisfying \eqref{eq17.53} and $a^{j1}\equiv 0,j=2,\ldots,d$.
Similarly, we introduce $\bL_2$ as the collection of divergence
form operators $\cL u=D_i(a^{ij}D_{j}u)$ satisfying
\eqref{eq17.53} and $a^{1j}\equiv 0,j=2,\ldots,d$. Notably, the adjoint operator of any operator in $\bL_1$ is in $\bL_2$, and the adjoint operator of any operator in $\bL_2$ is in $\bL_1$. In the
remaining part of this section as well as in the next section, we
assume $a^{ij}=a^{ij}(x^2)$.

For any $r>0$, denote $\tilde B_r^+$ to be the half ball of radius $r$ in $\bR^2$
and $\hat B_r$ to be the ball of radius $r$ in $\bR^{d-2}$. Both
of them are centered at the origin. We also denote $\Gamma_r=B_r\cap \partial\bR^d_+$.

We will use the following
special form of the Sobolev imbedding theorem, which follows by
applying the standard Sobolev imbedding theorem to $x'$ and $x''$
separately. For reader's convenience, we give a proof of it in the
appendix.
\begin{lemma}
                                                \label{lem0}
Let $\Omega=\tilde B_1^+\times \hat B_1$, $p_0\in (2,\infty)$ be a constant, and $k\ge d/2$ be an integer. Suppose that the function $f\in L_{p_0}(\Omega)$ satisfies
$$
D_{x''}^j f,\,\,D_{x''}^jD_{x'} f\in L_{p_0}(\Omega)
$$
for any $0\le j\le k$. Then we have $f\in C^{1-2/p_0}(\Omega)$ and
\begin{equation*}
\|f\|_{C^{1-2/{p_0}}(\Omega)}\\
\le N\sum_{j=0}^k \Big(\|D_{x''}^j
f\|_{L_{p_0}(\Omega)}+N\|D_{x''}^jD_{x'}f\|_{L_{p_0}(\Omega)}\Big)
\end{equation*}
for some constant $N=N(d,p_0)>0$.
\end{lemma}

Due to the lack of regularity of the coefficients $a^{ij}$ in $x^2$, the usual interior and boundary $C^{1,\alpha}$ estimates of solutions do not hold. Nevertheless, we prove the following boundary H\"older estimate for operators $\cL\in \bL_2$.

\begin{lemma}
                                    \label{lem2.1}
Let $\lambda\ge 0$ be a constant and $\cL\in \bL_2$. Suppose that $u\in C^\infty(\bar B_2^+)$ satisfies
$$
\cL u-\lambda u=0\quad \text{in}\,\,B_2^+
$$
with the conormal derivative boundary condition
$a^{11}D_1 u=0$ on $\Gamma_2$. Then there exist
constants $\alpha\in (0,1)$ and $N>0$ depending only on $d$ and
$\delta$ such that
$$
\|D_1 u\|_{C^\alpha(B_1^+)}+\|D_{x''}
u\|_{C^\alpha(B_1^+)}+\|a^{2j}D_j u\|_{C^\alpha(B_1^+)}
+\lambda^{1/2}\|u\|_{C^\alpha(B_1^+)}
$$
\begin{equation}
                                        \label{eq20.01}
\le N\|Du\|_{L_2(B_2^+)}+N\lambda^{1/2} \|u\|_{L_2(B_2^+)}.
\end{equation}
\end{lemma}
\begin{proof}
Denote $w_j=D_j u,j=1,\ldots,d$ and $y$ to be a point in $\bR^{d-1}$. First we consider the case when
$\lambda=0$.
%Note that since $a^{ij}=a^{ij}(x^2)$ and $D_1u=0$ on $B_2\cap\partial \bR^d_+$,
We claim that: i) $w_1$ satisfies $\cL w_1=0$ in $B_2^+$ with the
Dirichlet boundary condition $w_1=0$ on $\Gamma_2$, and ii)
$w_k,k>2$ satisfies $\cL w_k=0$ in $B_2^+$ with the conormal
derivative boundary condition $a^{11}D_1 w_k=0$ on $\Gamma_2$. To
verify the first claim, we take a function $\xi\in
C_0^\infty(\{|y|<2\})$ and an even function $\eta\in
C_0^\infty([-1,1])$ satisfying $\eta(0)=1$ and $\eta$ is
decreasing in $(0,1)$. We also take $\varepsilon>0$ sufficiently
small such that
$$
\psi(x):=\eta(x^1/\varepsilon)\xi(x^2,\ldots,x^d)\in C_0^\infty(B_2).
$$
Since $u$ satisfies $\cL u=0$ in $B_2^+$
with the conormal derivative boundary condition
$a^{11}D_1 u=0$ on $\Gamma_2$, it holds that
\begin{equation}
                                    \label{eq21.06}
\int_{B_2^+}a^{ij}D_ju D_i\psi \,dx=0.
\end{equation}
Note that, by the smoothness of $u$ and the conditions $a^{1j}=0$
for $j\ge 2$ and $a^{ij}=a^{ij}(x^2)$, we have
\begin{align*}
\lim_{\varepsilon\to 0}\int_{B_2^+}a^{1j}D_ju D_1\psi \,dx
&=\lim_{\varepsilon\to 0}\int_{B_2^+}a^{11}D_1u D_1\psi \,dx\\
&=-\int_{|y|<2}a^{11}D_1u(0,y) \xi(y)\,dy,
\end{align*}
and
$$
\lim_{\varepsilon\to 0}\int_{B_2^+}\sum_{i\ge 2}a^{ij}D_ju D_i\psi \,dx
=\lim_{\varepsilon\to 0}\int_{B_2^+}\sum_{i\ge 2}a^{ij}D_ju D_i\xi\eta(x^1/\varepsilon) \,dx=0.
$$
It then follows from \eqref{eq21.06} that
$$
\int_{|y|<2}a^{11}D_1u(0,y) \xi(y)\,dy=0.
$$
Since $\xi$ is an arbitrary function in $C_0^\infty(\{|y|<2\})$
and $a^{11}\ge \delta$, we obtain that $w_1=D_1 u=0$ on
$\Gamma_2$. Next, for any $\phi\in C_0^\infty(B_2\cap \bar
\bR^d_+)$ such that $\phi=0$ on $\Gamma_2$, integrating by parts
gives
\begin{align*}
\int_{B_2^+}a^{ij}D_j w_1 D_i\phi\,dx
&=-\int_{\Gamma_2}a^{ij}D_j u D_i\phi-\int_{B_2^+}a^{ij}D_j u D_{1i}\phi\,dx\\
&=-\int_{\Gamma_2}a^{1j}D_j u D_1\phi=0.
\end{align*}
Here in the second equality we used the fact that $u$ satisfies $\cL u=0$ in $B_2^+$
with the conormal derivative condition
$a^{11}D_1 u=0$ on $\Gamma_2$ as well as $D_j\phi=0$ on $\Gamma_2$ for $j\ge 2$. In the third equality we used $a^{1j}=0$ for $j\ge 2$ and $D_1 u=0$ on $\Gamma_2$. This complete the proof of the first claim. The second claim is obvious. Indeed, for any $\phi\in C_0^\infty(B_2\cap \bar \bR^d_+)$ and any $k=3,\ldots,d$, by using integration by parts we have
\begin{equation*}
\int_{B_2^+}a^{ij}D_j w_k D_i\phi\,dx
=-\int_{B_2^+}a^{ij}D_j u D_{ik}\phi\,dx=0.
\end{equation*}

Now by the DeGiorgi--Nash--Moser estimate,
\begin{equation}
                                        \label{eq14.38}
\|D_1u\|_{C^{\alpha_1}(B_1^+)}
\le N\|D_1u\|_{L_2(B_2^+)},
\end{equation}
and
\begin{equation}
                                        \label{eq14.49}
\|D_{x''}u\|_{C^{\alpha_1}(B_1^+)}
\le N\|D_{x''}u\|_{L_2(B_2^+)}
\end{equation}
for some $\alpha_1=\alpha_1(d,\delta)\in (0,1)$. Moreover, by the
reverse H\"older's inequality and the local $L_2$ estimate, for
some $p_0=p_0(d,\delta)>2$, we have
\begin{equation}
                                        \label{eq14.47}
\|DD_1 u\|_{L_{p_0}(B_1^+)}\le N(d,\delta)\|D_1 u\|_{L_{2}(B_{2}^+)},
\end{equation}
and
\begin{equation}
                                        \label{eq20.24}
\|DD_{x''} u \|_{L_{p_0}(B_1^+)}
\le N(d,\delta)\|D_{x''}u\|_{L_{2}(B_{2}^+)}.
\end{equation}
Thanks to the equation of $u$, we have
$$
D_2(a^{2j}D_j u)=-\sum_{k\neq 2}D_k(a^{kj}D_j u)=-\sum_{k\neq 2}a^{kj}D_{kj}u,
$$
which together with \eqref{eq14.47} and \eqref{eq20.24} yields
\begin{equation*}
                                           % \label{eq20.33}
\|D_2 (a^{2j}D_j u)\|_{L_{p_0}(B_1^+)}+\|D_1 (a^{2j}D_j u)\|_{L_{p_0}(B_1^+)}\le N\|D u\|_{L_{2}(B_{2}^+)}.
\end{equation*}
Then by using a standard scaling argument, for any $0<r<R\le 2$, we have
\begin{equation}
                                            \label{eq20.33bb}
\|D_{x'} (a^{2j}D_j u)\|_{L_{p_0}(B_r^+)}\le N\|D
u\|_{L_2(B_{R}^+)},
\end{equation}
where $N=N(d,\delta,r,R)>0$.

For  any integer $l=1,2,\ldots, [d/2]+1$, take a increasing sequence
$3/2<r_1<r_2<\ldots<r_{l+1}<2$. Since $D_{x''}^l
u$ satisfies the same equation as $u$ with the same boundary
condition, by applying \eqref{eq20.33bb} to $D_{x''}^l u$ and
using the local $L_2$ estimate repeatedly, we get
\begin{multline}
                                        \label{eq20.24x}
\|D_{x''}^l D_{x'} (a^{2j}D_j
u)\|_{L_{p_0}(B_{3/2}^+)}+\|D_{x''}^l(a^{2j}D_j
u)\|_{L_{p_0}(B_{3/2}^+)}
\le N\|D D_{x''}^l u \|_{L_{2}(B_{r_1}^+)}\\
\le N\|D D_{x''}^{l-1} u \|_{L_{2}(B_{r_2}^+)}
\le\ldots
\le N(d,\delta,k)\|Du\|_{L_{2}(B_{2}^+)}.
\end{multline}
Since $\alpha_2:=1-2/p_0>0$ and
$$
B_1^+\subset \tilde B_1^+\times\hat B_1\subset B_{3/2}^+,
$$
it follows from \eqref{eq20.24x} and Lemma \ref{lem0} that
\begin{equation}
                                    \label{eq15.01}
\|a^{2j}D_j u\|_{C^{\alpha_2}(B_1^+)}\le N\|D u\|_{L_{2}(B_{2}^+)}.
\end{equation}
Collecting \eqref{eq14.38}, \eqref{eq14.49}, and \eqref{eq15.01} gives \eqref{eq20.01} with $\alpha=\min\{\alpha_1,\alpha_2\}$.

Next we treat the case $\lambda>0$ by adapting an idea by S.
Agmon. Introduce a new variable $y\in \bR$ and define
$$
v(x,y)=u(x)\big(\cos(\lambda^{1/2} y)+\sin(\lambda^{1/2} y)\big).
$$
For $r>0$, denote $\bB_r$
and $\mathbf{B}_r^+$ to be the $d+1$ dimensional ball and half ball
with radius $r$ centered at the origin. Then $v$ satisfies
$$
\tilde \cL v=0 \quad \text{in}\,\,\bB_2^+
$$
with the conormal boundary condition $a^{11}D_1 v=0$ on $\bB_2\cap
\partial\bR^{d+1}_+$. Here $\tilde \cL=\cL+D_y^2\in \bL_2$ is a
divergence elliptic operator in $\bR^{d+1}$. This reduces the
problem to the case $\lambda=0$. By the proof above, we have
$$
\|D_1 v\|_{C^\alpha(\bB_1^+)}+\|D_{x''} v\|_{C^\alpha(\bB_1^+)}+\|D_y v\|_{C^\alpha(\bB_1^+)}+\|a^{2j}D_j v\|_{\tilde C^\alpha(\bB_1^+)}
$$
\begin{equation}
                                        \label{eq20.01z}
\le N\|Dv\|_{L_2(\bB_2^+)}.
\end{equation}
Now we observe that $D_1u$, $D_{x''}u$, $\lambda^{1/2} u$, and
$a^{2j}D_j u$ are restrictions of $D_1v$, $D_{x''}v$, $D_y v$, and
$a^{2j}D_j v$ on the hyperplane $y=0$, respectively. Therefore,
the left-hand side of \eqref{eq20.01} is less than or equal to
that of \eqref{eq20.01z}. On the other hand, $Dv$ is a linear
combination of
$$
Du(x)\big(\cos(\lambda^{1/2} y)+\sin(\lambda^{1/2} y)\big),\quad
\lambda^{1/2} u(x)\big(-\sin(\lambda^{1/2} y)+\cos(\lambda^{1/2}
y)\big).
$$
So the
right-hand side of \eqref{eq20.01z} is less than that of
\eqref{eq20.01}. The lemma is proved.
\end{proof}

For operators $\cL\in \bL_1$, there is a similar estimate under the Dirichlet boundary condition.
\begin{lemma}
                                    \label{lem2.2}
Let $\lambda\ge 0$ be a constant and $\cL\in \bL_1$. Suppose that $u\in C^\infty(\bar B_2^+)$ satisfies
$$
\cL u-\lambda u=0\quad \text{in}\,\,B_2^+
$$
with the Dirichlet boundary condition $u=0$ on $\Gamma_2$. Then \eqref{eq20.01} holds
for some constants $\alpha\in (0,1)$ and $N>0$ depending only on
$d$ and $\delta$.
\end{lemma}
\begin{proof}
As in the proof of Lemma \ref{lem2.1}, it suffices to prove
\eqref{eq20.01} when $\lambda=0$. We define $w_j,j=1,\ldots,d$ as
before. It is easily seen that now $w_j,j>2$ satisfies $\cL w_j=0$
in $B_2^+$ and $w_j=0$ on $\Gamma_2$. We claim that $w_1$
satisfies $\cL w_1=0$ in $B_2^+$ with the conormal derivative
boundary condition $a^{1j}D_j w_1=0$ on $\Gamma_2$. For any
$\phi\in C_0^\infty(B_2\cap \bar \bR^d_+)$, we decompose it as
$$
\phi=\phi_1+\phi_2,\quad
\phi_1=\eta(x^1/\varepsilon)\phi(0,x^2,\ldots,x^d)+\big(1-\eta(x^1/\varepsilon)\big)\phi,\quad
$$
$$
\phi_2=\eta(x^1/\varepsilon)\big(\phi-\phi(0,x^2,\ldots,x^d)\big),
$$
where $\eta$ is the function defined in the proof of Lemma \ref{lem2.1} and $\varepsilon>0$ is sufficiently small such that $\phi_1\in C_0^\infty(B_2\cap \bar \bR^d_+)$. Integration by parts gives
\begin{align}
\int_{B_2^+}a^{ij}D_j w_1 D_i\phi_1\,dx
&=-\int_{\Gamma_2}a^{ij}D_j u D_{i}\phi_1
-\int_{B_2^+}a^{ij}D_j u D_{i}(D_1\phi_1)\,dx\nonumber\\
&=-\int_{\Gamma_2}a^{ij}D_j u D_{i}\phi_1=0
                                \label{eq00.24}
\end{align}
Here in the second equality we used the fact that  $u$ satisfies $\cL u=0$ in $B_2^+$ with the Dirichlet boundary condition $u=0$ on $\Gamma_2$ and $D_1\phi_1=0$ on $\Gamma_2$. In the third equality, we used $D_j u=0,j\ge 2$ on $\Gamma_2$, $a^{i1}=0$ for $i\ge 2$ and $D_1\phi_1=0$ on $\Gamma_2$. Note that $\phi_2$ vanishes for any $x^1>\varepsilon$ and $|D\phi_2|\le N$ where $N$ is independent of $\varepsilon$. Therefore,
\begin{equation}
                        \label{eq00.25}
\lim_{\varepsilon\to 0}\int_{B_2^+}a^{ij}D_j w_1 D_{i}\phi_2\,dx
=\lim_{\varepsilon\to 0}\int_{B_2^+\cap\{x^1\le \varepsilon\}}a^{1j}D_{j1} u D_{i}\phi_2\,dx=0.
\end{equation}
Combining \eqref{eq00.24} and \eqref{eq00.25}, we conclude
$$
\int_{B_2^+}a^{ij}D_j w_1 D_i\phi\,dx=0,
$$
which completes the proof of the claim.
Therefore, we can still use the
DeGiorgi--Nash--Moser estimate to bound the first two terms on the
left-hand side of \eqref{eq20.01}. The third term on the left-hand
side of \eqref{eq20.01} is estimated in exactly the same way as in
the proof of Lemma \ref{lem2.1}. We omit the details.
\end{proof}

\begin{remark}
                                    \label{rem2.4}
The results of Lemmas \ref{lem2.1} and \ref{lem2.2} still hold
true when $B_1^+$ and $B_2^+$ are replaced with $B_1(x_0)$ and
$B_2(x_0)$ provided that $B_2(x_0)\subset \bR^d_+$. For these
interior estimates, the condition $\cL\in \bL_2$ or $\cL\in \bL_1$
is not needed, since there is no boundary condition in this case.
\end{remark}

Denote
$$
U=U(x):=|D_1 u|+|D_{x''}u|+|a^{2j}D_j u|+\lambda^{1/2}|u|.
$$
Because $a^{22}\ge \delta>0$ and $a^{ij}$ are bounded, it is easily seen that
\begin{equation}
                                        \label{eq19.20.22}
N^{-1} \big(|Du|+\lambda^{1/2} |u|\big) \le U\le
N\big(|Du|+\lambda^{1/2} |u|\big),
\end{equation}
where $N=N(d,\delta)>0$.
\begin{corollary}
                                    \label{cor2.3}
Let $\lambda> 0$ be a constant, $\cL\in \bL_2$, and
$g=(g^1,\ldots,g^d),f\in C^\infty_{\text{loc}}(\bar\bR^d_+)$. Suppose that
$u\in C^\infty_{\text{loc}}(\bar \bR^d_+)$ satisfies
$$
\cL u-\lambda u=\Div g+f\quad \text{in}\,\,\bR^d_+
$$
with the conormal derivative boundary condition $a^{11}D_1 u=g^1$
on $\partial\bR^d_+$. Then, for any $r>0$, $\kappa\ge 32$ and
$x_0\in \bar \bR^d_+$, we have
\begin{multline}
                                                  \label{eq17.21.22}
\dashint_{B^{+}_r(x_0)}\dashint_{B^{+}_r(x_0)}|U (x)-
U (y)|^2\,dx\,dy\\
\leq N\kappa^d \dashint_{B^{+}_{\kappa r} (x_0)}
\big(|g|^2+\lambda^{-1}f^2\big) \,dx
 +N\kappa^{-2\alpha} \dashint_{B^{+}_{\kappa r}(x_0)}U^2\,dx,
\end{multline}
where $\alpha$ is the constant from Lemma \ref{lem2.1} and the
constant $N$ depends only on $d$ and $\delta$.

The same estimate
holds for $\cL\in \bL_1$ if the conormal derivative boundary
condition is replaced with the Dirichlet boundary condition $u=0$
on $\partial\bR_+^d$.
\end{corollary}
\begin{proof}
By standard mollifications, we may assume $a^{ij}\in C^\infty$.
Dilations show that it suffices to prove the lemma only for
$\kappa r=8$. After a shift of the coordinates, we may assume
$x_0=(x_0^1,0,\ldots,0)$. We consider two cases: i) $x_0^1<1$,
i.e., when $x_0$ is close to the boundary $\partial\bR^d_+$;  ii)
$x_0^1\ge 1$, i.e., when $x_0$ is away from the boundary.

{\em Case i).}
Since $r=8/\kappa\le 1/4$, we have
\begin{equation}
                                    \label{eq15.22.55s}
B^{+}_r(x_0)\subset B^{+}_2 \subset B^{+}_6\subset B^{+}_{\kappa r}(x_0).
\end{equation}
%By using a standard density argument, we may assume $a^{ij}\in
%C^\infty(\bR)$.
Take a smooth cutoff function $\eta\in C_0^\infty(B_6)$ such that
$\eta\equiv 1$ in $B_4$ and $0\le \eta\le 1$  in $B_6$. By using
the classical $W^1_2$-solvability for divergence elliptic
equations, there exists a unique solution $w\in W^1_2(\bR^d_+)$ to
the equation
$$
\cL w-\lambda w=\Div(\eta g)+ \eta f\quad \text{in}\,\,\bR^d_+
$$
with the conormal derivative boundary condition $a^{11}D_1 u=\eta
g^1$ on $\partial \bR^d_+$. Moreover, we have
$$
\|D w\|_{L_2(\bR^d_+)}+\lambda^{1/2} \|w\|_{L_2(\bR^d_+)} \le
N(d,\delta)\big(\|\eta g\|_{L_2(\bR^d_+)}+\lambda^{-\frac 1
2}\|\eta f\|_{L_2(\bR^d_+)}\big),
$$
which implies that
\begin{align}
                                                  \label{eq15.22.38s}
\dashint_{B^{+}_r(x_0)}|D w|^2+\lambda w^2 \,dx
&\leq N\kappa^d \dashint_{B^{+}_{\kappa r}
(x_0)} |g|^2+\lambda^{-1}f^2 \,dx,\\
                                        \label{eq15.22.40s}
\dashint_{B^{+}_r(\kappa x_0)}| D w|^2+\lambda w ^2 \,dx &\leq N
\dashint_{B^{+}_{\kappa r}(x_0)} |g|^2+\lambda^{-1}f^2 \,dx.
\end{align}
By the classical elliptic theory, $w\in C^\infty(\bar\bR^d_+)$.
Now let $v=u-v\in C_b^{\infty}(\bar{B}_{\kappa r}^{+}(x_0))$, which clearly satisfies
$$
\cL v-\lambda v=0\quad \text{in}\,\,B_4^+
$$
and $a^{11}D_1 v=0$ on $\Gamma_4$. We define
$$
V=V(x):=|D_1 v|+|D_{x''}v|+|a^{2j}D_j v|+\lambda^{1/2}|v|.
$$
Recall that $r=8/\kappa$. By Lemma \ref{lem2.1}, the triangle inequality, and \eqref{eq15.22.55s},
\begin{align}
                                                  \label{eq15.22.57s}
&\dashint_{B^{+}_r(x_0)}\dashint_{B^{+}_r(x_0)}|V(x)-
 V(y)|^2 \,dx\,dy\nonumber\\
&\,\,\le Nr^{ 2\alpha}\big([D_1 v]_{C^\alpha(B_2^+)}
+[D_{x''} v]_{C^\alpha(B_2^+)}+[a^{2j}D_j v]_{C^\alpha(B_2^+)}
+\lambda^{1/2} [v]_{C^\alpha(B_2^+)}\big)^2 \nonumber\\
&\,\,\le N\kappa^{-2\alpha} \dashint_{B^{+}_{\kappa r}(x_0)}
 |  Dv|^2 +\lambda v^2\,dx\nonumber\\
&\,\,\le N\kappa^{-2\alpha} \dashint_{B^{+}_{\kappa r}(x_0)}
 V^2\,dx.
\end{align}
In the last inequality above, we used an inequality similar to \eqref{eq19.20.22} with $v$ and $V$ in place of $u$ and $U$.
Since $u=v+w$, combining \eqref{eq15.22.38s}, \eqref{eq15.22.40s}, \eqref{eq15.22.57s}, and the triangle inequality, we immediately get \eqref{eq17.21.22}.

{\em Case ii).} The proof of this case is similar and actually simpler. Recall that $\kappa r=8$ and $\kappa \ge 32$. We have
$$
B_r^+(x_0)=B_r(x_0)\subset B_{\kappa r/8}(x_0)\subset \bR^d_+.
$$
Now we take $\eta\in C_0^\infty(B_{\kappa r/8}(x_0)$ such that
$$
\eta\equiv 1\quad \text{in}\,\, B_{\kappa r/16}(x_0),\quad 0\le \eta\le 1\quad \text{in}\,\,B_{\kappa r/8}(x_0).
$$
We then follow the above proof and use the interior estimates instead of the boundary estimate (cf. Remark \ref{rem2.4}) to get
\begin{multline*}
                                       %           \label{eq17.21.22}
\dashint_{B_r(x_0)}\dashint_{B_r(x_0)}|U (x)-
U (y)|^2\,dx\,dy\\
\leq N\kappa^d \dashint_{B_{\kappa r/8}
(x_0)} |g|^2+\lambda^{-1}f^2 \,dx
 +N\kappa^{-2\alpha} \dashint_{B_{\kappa r/8}(x_0)}U^2\,dx,
\end{multline*}
which clearly yields \eqref{eq17.21.22}

%When $\lambda=0$, we take $\lambda_1>0$. Since
%$$
%\cL u-\lambda_1 u=\Div g+f-\lambda_1 u\quad \text{in} \,\,B_{\kappa r}^+(x_0),
%$$
%by the proof above, we have
%\begin{multline}
%                                                  \label{eq17.21.22}
%\dashint_{B^{+}_r(x_0)}\dashint_{B^{+}_r(x_0)}|U (x)-
%U (y)|^2\,dx\,dy\\
%\leq N\kappa^d \dashint_{B^{+}_{\kappa r}
%(x_0)} g^2+\lambda^{-1}(f^2-\lambda_1^2 u^2) \,dx
% +N\kappa^{-2\alpha} \dashint_{B^{+}_{\kappa r}(x_0)}U^2\,dx,
%\end{multline}
%where the constant $N$ is independent of $\lambda_1$.
%Taking the limit as $\lambda_1\to 0$ yields the desired estimate.

The proof of the last assertion is the same by using Lemma \ref{lem2.2} in place of Lemma \ref{lem2.1}. The details are thus omitted.
\end{proof}

In the measure
space $\bR^{d}_{+}$ endowed with the Borel $\sigma$-field and Lebesgue measure, we consider the filtration of dyadic cubes $\{\bC_{l},l\in\bZ\}$, where
$\bZ=\{0,\pm1,\pm2,...\}$ and $\bC_l$ is the collection of cubes
$$
(i_{1}2^{-l},(i_{1}+1)2^{-l}]\times...\times
(i_{d}2^{-l},(i_{d}+1)2^{-l}],
$$
where $i_{1},...,i_{d}\in\bZ,\,\,i_1\ge 0$. Let $\cC$ be the union of $\bC_l,l\in \bZ$.
Notice that if $x_0\in C\in \bC_l$, then for the smallest $r>0$
such that $C\subset B_{r}(x_0)$ we have
\begin{equation}
                                \label{eq20.20.20}
\dashint_{C} \dashint_{C}|g(y)-g(z)|
\,dy\,dz\leq N(d)
\dashint_{B^{+}_{r}(x_0)} \dashint_{B^{+}_{r}(x_0)}|g(y)-g(z)|
\,dy\,dz.
\end{equation}
On the other hand, for any $x_0\in \bar\bR^d_+$ and $r>0$, let $C\subset \cC$ be the smallest cube which contains $B_r^+(x_0)$. Then,
\begin{equation}
                                \label{eq20.20.20b}
\dashint_{B^{+}_{r}(x_0)} |g(y)|
\,dy\le N(d)\dashint_{C} |g(y)|
\,dy.
\end{equation}

For a function $g\in L_{1,\text{loc}}(\bR^{d}_+)$,
we define the maximal and sharp function of $g$ are given by
\begin{align*}
\cM  g (x) &= \sup_{C \in \cC, x\in C} \dashint_{C} |g(y)| \, dy,\\
g^{\#}(x) &= \sup_{C \in \cC, x \in C} \dashint_{C} \dashint_{C}|g(y) -
g(z)| \, dy\,dz.
\end{align*}
Let $p\in (1,\infty)$. By the Fefferman--Stein theorem on sharp
functions and the Hardy-Littlewood maximal function theorem, we
have
\begin{align}
                                    \label{eq20.20.41}
\| g \|_{L_p(\bR^{d}_+)} &\le N \| g^{\#} \|_{L_p(\bR^{d}_+)},\\
                                    \label{eq20.20.48}
\| \cM  g \|_{L_p(\bR^{d}_+)} &\le N \| g\|_{L_p(\bR^{d}_+)},
\end{align}
if $g \in L_p(\bR^{d}_+)$, where $N = N(d,p)>0$.

The following two propositions are the main results of this section.

\begin{proposition}
                                \label{prop0}
Suppose either $\cL\in \bL_1$ and $p\in (1,2]$ or $\cL\in \bL_2$
and $p\in [2,\infty)$. Let $g=(g^1,\ldots,g^d), f\in
C^\infty_0(\bar\bR^d_+)$. Then for any $\lambda> 0$ and $u\in
C^\infty_{\text{loc}}(\bar\bR^d_+)$ satisfying
\begin{equation}
                                            \label{eq22.22}
\cL u-\lambda u=\Div g+f\quad \text{in}\,\, \bR^d_+
\end{equation}
with the conormal derivative boundary condition $a^{1j}D_ju=g^1$
on $x^1=0$, we have
\begin{equation}
                \label{eq19.46}
\lambda\|u\|_{L_{p}(\bR^d_+)}+\lambda^{1/2}
\|Du\|_{L_{p}(\bR^d_+)}\le N\lambda^{1/2}\|g\|_{L_{p}(\bR^d_+)}
+N\|f\|_{L_{p}(\bR^d_+)},
\end{equation}
where $N=N(d,\delta,p)>0$. Furthermore, for $\lambda=0$ and $f\equiv 0$, we have
\begin{equation}
                \label{eq19.46b}
\|Du\|_{L_{p}(\bR^d_+)}\le N\|g\|_{L_{p}(\bR^d_+)}.
\end{equation}
\end{proposition}
\begin{proof}
Again we may assume $a^{ij}\in C^\infty$. 
By a duality argument, we only need to consider the case when
$\cL\in \bL_2$ and $p\in [2,\infty)$. The result is classical if
$p=2$. In the sequel, we suppose $p>2$. Since $\cL\in \bL_2$, the
conormal derivative condition becomes $D_1 u=g^1$ on $x^1=0$. Now
due to \eqref{eq20.20.20}, \eqref{eq20.20.20b}, and Corollary
\ref{cor2.3}, we obtain a pointwise estimate for $U$:
\begin{equation}
                                        \label{eq20.20.33}
U^\#(x_0)\le N\kappa^{\frac
d2}\Big(\cM\big(|g|^2+\lambda^{-1}f^2\big)(x_0)\Big)^{\frac 12}+
N\kappa^{-\alpha}\Big(\cM\big(U^2\big)(x_0)\Big)^{\frac 12}
\end{equation}
for any $x_0\in \bar\bR^d_+$. It follows from \eqref{eq20.20.33}, \eqref{eq20.20.41}, and \eqref{eq20.20.48} that
\begin{align*}
\|U&\|_{L_p(\bR^d_+)}\le N\big\|U^\#\big\|_{L_p(\bR^d_+)}\\
&\le N\kappa^{\frac
d2}\big\|\cM\big(|g|^2+\lambda^{-1}f^2\big)\big\|^{1/2}_{L_{p/2}(\bR^d_+)}
+N\kappa^{-\alpha}\big\|\cM\big(U^2\big)\big\|_{L_{\frac p2}(\bR^d_+)}^{\frac 1 2}\\
&\le N\kappa^{\frac d2}\big\||g|^2+\lambda^{-1}f^2\big\|^{\frac
12}_{L_{\frac p2}(\bR^d_+)} +N\kappa^{-\alpha}\big\|U^2\big\|_{L_{\frac
p2}(\bR^d_+)}^{\frac 12},
\end{align*}
which immediately yields \eqref{eq19.46} upon taking $\kappa$ sufficiently large.
Next we treat the case when $\lambda=0$ and $f\equiv 0$. Note that for any $\lambda_1>0$, we have
$$
\cL u-\lambda_1 u=\Div g-\lambda_1 u\quad \text{in}\,\, \bR^d_+.
$$
From \eqref{eq19.46}, we have
\begin{equation}
                \label{eq19.46c}
\lambda_1\|u\|_{L_{p}(\bR^d_+)}+\lambda_1^{1/2}
\|Du\|_{L_{p}(\bR^d_+)}\le N\lambda^{1/2}_1\|g\|_{L_{p}(\bR^d_+)}
+N\lambda_1\|u\|_{L_{p}(\bR^d_+)},
\end{equation}
where $N$ is independent of $\lambda_1$. Dividing both sides of
\eqref{eq19.46c} by $\lambda^{1/2}_1$ and taking the limit as
$\lambda_1\to 0$ give \eqref{eq19.46b}. This completes the proof
of the theorem.
\end{proof}

\begin{proposition}
                                \label{prop1}
Suppose either $\cL\in \bL_1$ and $p\in [2,\infty)$ or $\cL\in
\bL_2$ and $p\in (1,2]$. Let $g=(g^1,\ldots,g^d), f\in
C^\infty_0(\bar\bR^d_+)$. Then for any $\lambda> 0$ and $u\in
C^\infty_0(\bar\bR^d_+)$ satisfying \eqref{eq22.22} and $u=0$ on
$x^1=0$, we have \eqref{eq19.46}, where $N=N(d,\delta,p)>0$.
Furthermore, for $\lambda=0$ and $f\equiv 0$, we have
\eqref{eq19.46b}.
\end{proposition}
\begin{proof}
The proof is the same as the proof of Proposition \ref{prop0} by using the last assertion of Corollary \ref{cor2.3}. We thus omit the details.
\end{proof}

\mysection{Main theorems and proofs}
                                                \label{sec3}

First we consider elliptic equations in the half space with Dirichlet boundary condition.
Recall that $a^{ij}=a^{ij}(x^2)$ satisfies \eqref{eq17.53}. For any domain $\Omega\subset\bR^d$ and $p\in (1,\infty)$, we define $\WO_p^1(\Omega)$ to be the completion of $C_0^\infty(\Omega)$ in the $W_p^1(\Omega)$ space, and $\WO_p^2(\Omega)=\WO_p^1(\Omega)\cap W^2_p(\Omega)$.
\begin{theorem}
                                        \label{thm1}
Let $p\in (1,2]$. Then for any $\lambda\ge 0$ and $u\in \WO^{2}_p(\bR^d_+)$, we have
\begin{equation}
                \label{eq10.30}
\lambda\|u\|_{L_{p}(\bR^d_+)}+\lambda^{1/2}
\|Du\|_{L_{p}(\bR^d_+)}+\|D^2u\|_{L_{p}(\bR^d_+)}\le N\|Lu-\lambda
u\|_{L_{p}(\bR^d_+)},
\end{equation}
where $N=N(d,\delta,p)>0$.
Moreover, for any $f\in L_p(\bR^d_+)$ and $\lambda>0$ there is a unique $u\in W^2_p(\bR^d_+)$ solving
$$
Lu-\lambda  u=f\quad \text{in} \,\, \bR^d_+
$$
with the Dirichlet boundary condition $u=0$ on $\partial \bR^d_+$.
\end{theorem}

\begin{remark}
                            \label{rm5.46}
To see that the range of $p$ in Theorem \ref{thm1} is sharp, we consider the following example. Let $d=2$, $\theta\in (\pi/2,\pi)$ and $\eta\in C_0^\infty((-2,2))$ be a cutoff function satisfying $\eta\equiv 1$ on $[-1,1]$. In the polar coordinates, define
$$
u(r,\omega):=r^{\pi/\theta}\sin(\omega \pi/\theta)\eta(r),\quad f:=\Delta u-\lambda u.
$$
It is easily seen that $u$ satisfies the Dirichlet boundary condition in the angle $\Omega_\theta:=\{\omega\in (0,\theta)\}$ and $f\in L_p(\Omega_\theta)$ with $p=2/(2-\pi/\theta)$. However, $D^2 u\notin L_p(\Omega_\theta)$. Now we take a linear transformation to map $\Omega_\theta$ to the first quadrant $\{(y^1,y^2)\in\bR^2\,|\,y^1>0,y^2>0\}$. Let $\tilde u$ and $\tilde f$ be the functions in the $y$-coordinates and $a^{ij}$ be corresponding constant coefficients. We take the odd extensions of $\tilde u$ and $\tilde f$ with respect to $y^2$, and denote
$$
\tilde a^{ij}(y^2)={\text{sgn}(y^2)}a^{ij}\quad \text{for}\,\,i\neq j,\quad
\tilde a^{ij}(y^2)=a^{ij}\quad \text{otherwise}.
$$
Clearly, $\tilde u$ satisfies
$$
\tilde a^{ij}D_{ij}\tilde u(y)-\lambda \tilde u(y)=\tilde f(y) \quad \text{in}\,\,\bR^2_+
$$
and $\tilde u=0$ on $\partial\bR^2_+$. Moreover, we have $\tilde f\in L_p(\bR^2_+)$, but $D^2 \tilde u\notin L_p(\bR^2_+)$. Note that $p\searrow 2$ as $\theta\nearrow \pi$. Thus, this example implies that the range of $p$ in the above theorem is actually sharp in the sense that for any $p>2$ there is an elliptic operator $L=\tilde a^{ij}(y^2)D_{ij}$ with the ellipticity constant depending on $p$, such that the $W^2_p$ estimate does not hold for $L$.
\end{remark}

\begin{proof}[Proof of Theorem \ref{thm1}]
Let $f=Lu-\lambda u$. By mollifications and a density argument, we
may assume $u\in C_0^\infty(\bar \bR^d_+)$ and $u\equiv 0$ on
$x^1=0$. We move the mixed second derivatives $a^{1j}D_{1j}u$ and
$a^{j1}D_{j1}u,j=2,\ldots,d$ to the right-hand side to get
$$
a^{11}D_{11}u+\sum_{i,j=2}^d a^{ij}(x^2)D_{ij}u-\lambda u=f-\sum_{j=2}^d(a^{1j}+a^{j1})D_{1j}u.
$$
Since the $W^2_p$-estimate in the whole space is available
when the coefficients depend only on one direction (cf. \cite{D_TAMS11,KimKrylov07}),
by using odd extensions of $u$ and $f$ with respect to $x^1$, we
get
$$
\lambda\|u\|_{L_{p}(\bR^d_+)}+\|D^2u\|_{L_{p}(\bR^d_+)}\le N\|f\|_{L_{p}(\bR^d_+)}
+N\|DD_{1}u\|_{L_{p}(\bR^d_+)}.
$$
Therefore, in order to prove \eqref{eq10.30}, it suffices to show
\begin{equation}
                \label{eq10.36}
\|DD_{1}u\|_{L_{p}(\bR^d_+)}\le N\|f\|_{L_{p}(\bR^d_+)}.
\end{equation}
%Dividing both sides by $a^{11}$ and using the maximum principle reduce the general situation to the case $a^{11}=1$.
%Denote $v=D_2 u$.

We adapt an idea in \cite{D_TAMS11} to rewrite $L$ into a divergence form operator by making a change of variables:
$$
y^2=\phi(x^2):=\int_0^{x^2} \frac 1 {a^{22}(s)}\,ds,\quad y^j=x^j,\,\,j\neq 2.
$$
It is easy to see that $\phi$ is a bi-Lipschitz map and
\begin{equation*}
                                            %\label{eq11.51}
\delta \le y^2/x^2\le \delta^{-1},\quad D_{y^2}=a^{22}(x^2)D_{x^2}.
\end{equation*}
Denote
\begin{align}
v(y)&=u(y_1,\phi^{-1}(y^2),y''),\nonumber\\
\tilde f(y)&=f(y^1,\phi^{-1}(y^2),y'),\nonumber\\
\tilde a^{22}(y^2)&=1/a^{22}(\phi^{-1}(y^2)),\quad \tilde a^{2j}\equiv 0,\,\,j\neq 2,\label{eq10.41}\\
\tilde a^{j2}(y^2)&=\big((a^{2j}+ a^{j2})/a^{22}\big)(\phi^{-1}(y^2)),\,\,j\neq 2,\nonumber\\
\tilde a^{1j}(y^2)&=(a^{1j}+ a^{j1})(\phi^{-1}(y^2)),\quad \tilde a^{j1}\equiv 0,\,\,j> 2,\nonumber
\end{align}
and $\tilde a^{ij}(y^2)=a^{ij}(\phi^{-1}(y^2))$ for the other $(i,j)$.
In the $y$--coordinates, define a divergence form operator $\cL$ by
$$
\cL v=D_i(\tilde a^{ij} D_j v).
$$
It is easily seen that $\cL\in \bL_1$ and is uniformly nondegenerate with an ellipticity constant depending only on $\delta$. A simple calculation shows that $v$ satisfies in $\bR^{d}_+$
\begin{equation}
                                    \label{eq18.47}
\cL v-\lambda v=\tilde f.
\end{equation}
Moreover, the condition $u\equiv 0$ on $x^1=0$ implies
\begin{equation}
                                    \label{eq18.28}
\tilde a^{1j}D_{1j}v=\tilde f\quad \text{on}\,\,\{x^1=0\}.
\end{equation}
By differentiating \eqref{eq18.47} with respect to $x^1$ and using \eqref{eq18.28}, we get that $w:=D_1 v$ satisfies
$$
\cL w-\lambda w=D_1 \tilde f
$$
in $\bR^d_+$ with the conormal derivative boundary condition $\tilde a^{1j}D_{j}w=\tilde f$ on $x^1=0$.
We then deduce \eqref{eq10.36} from Proposition \ref{prop0}. The theorem is proved.
\end{proof}

\begin{remark}
                                        \label{rem3.2}
In \cite{DK_TAMS10} it was proved that the $W^2_p$-solvability in $\bR^d$ for elliptic equations holds when $p\in [2,2+\varepsilon)$ and the leading coefficients are measurable in two directions and independent of (or VMO with respect to) the orthogonal directions. Here $\varepsilon>0$ is a constant which depends only on $d$ and $\delta$. Thus by using the method of odd/even extensions, the result of Theorem \ref{thm1} holds for $p\in (1,2+\varepsilon)$. Similarly, in Theorems \ref{thm3} and \ref{thm5.1} below, the range of $p$ can be extended to $(1,2+\varepsilon)$.
%The same remark applies to Theorems \ref{thm2} and \ref{thm4}, in which the range of $p$ can be extended to $(2-\varepsilon,\infty)$.
\end{remark}

Next we consider the Neumann boundary problem.

\begin{theorem}
                                        \label{thm2}
Let $p\in [2,\infty)$. Then for any $\lambda\ge 0$ and $u\in W^{2}_p(\bR^d_+)$ satisfying $D_1 u=0$ on $x^1=0$, we have
\begin{equation*}
                %\label{eq10.30n}
\lambda\|u\|_{L_{p}(\bR^d_+)}+\lambda^{1/2}
\|Du\|_{L_{p}(\bR^d_+)}+\|D^2u\|_{L_{p}(\bR^d_+)}\le N\|Lu-\lambda
u\|_{L_{p}(\bR^d_+)},
\end{equation*}
where $N=N(d,\delta,p)>0$.
Moreover, for any $f\in L_p(\bR^d_+)$ and $\lambda>0$ there is a unique $u\in W^2_p(\bR^d_+)$ solving
$$
Lu-\lambda  u=f\quad \text{in}\,\, \bR^d_+
$$
with the Neumann boundary condition $D_1 u=0$ on $\partial \bR^d_+$.
\end{theorem}
\begin{proof}
As in the proof of Theorem \ref{thm1}, we may assume that $u\in
C_0^\infty(\tilde \bR^d_+)$ and $D_1 u=0$ on $x^1=0$. Since the $W^2_p$-estimate in the whole space is available
when the coefficients depend only on one direction (cf. \cite{D_TAMS11,KimKrylov07}),
by using even extensions of $u$ and $f$ with respect to $x^1$, we again
only need to show \eqref{eq10.36}. By the same change of
variables, we find that $v$ satisfies
$$
\cL v-\lambda v= \tilde f
$$
in $\bR^d_+$ and $D_1 v=0$ on $x^1=0$, where $v$, $\tilde f$ and $\cL\in \bL_1$ are defined in the proof of Theorem \ref{thm1}. Thus, $w_1:=D_1 v$ satisfies
$$
\cL w_1-\lambda w_1= D_1\tilde f
$$
in $\bR^d_+$ and $w_1=0$ on $x^1=0$. To finish the proof of \eqref{eq10.36}, it suffices to use Proposition \ref{prop1}.
%Alternatively, one can prove \eqref{eq10.36} by estimating the mean oscillation of $D_{12}u$. To this end, assume $a^{11}=1$. If $u$ satisfies $Lu-\lambda u=0$ and $D_1 u=0$ on $x^1=0$, then $w=D_{12}u$ satisfies $\cL w-\lambda w=0$ and $w=0$ on $x^1=0$, where $\cL$ is defined below \eqref{eq10.36}. The mean oscillation estimate of $w$ then follows from the DeGiorgi--Nash--Moser estimate.
\end{proof}

\begin{remark}
In contrast to Theorem \ref{thm1}, it remains unclear to us whether the results in Theorem \ref{thm2} are true for $p\in (1,2)$.
\end{remark}

%\begin{remark}
%Theorem \ref{thm2} can be extended to higher dimensional case (as well as equations with coefficients measurable in $x^2$ and VMO in $x^1$ and $x''$, where $x''=(x^3,\ldots,x^d)$.) Indeed, the $W^2_2$ estimate follows from the technique of even/odd extensions and the $W^2_2$ estimate in the whole space when the coefficients depend only on two variables; cf. \cite{DK_TAMS10}. Then by differentiating the equation and applying the Krylov--Safonov theorem, we can estimate the H\"older norm of $D_{jk} u$ for $j,k\ge 3$ by the $L^2$ norm of $D^2u$, provided that $u$ satisfies
%$Lu-\lambda u=0$ in $\bR^d_+$ and $D_1 u=0$ on $x^1=0$. Finally, we move all the second derivatives $D_{jk}u$ and $D_{kj}u$ for $j\ge 1,k\ge 3$ to the right-hand side of the equation and apply the result in Theorem \ref{thm2}.
%\end{remark}

\mysection{Equations with more general coefficients}
                                                            \label{sec4}

In this section, we consider second-order elliptic equations
$$
Lu-\lambda u:=a^{ij}D_{ij}u+b^iD_iu+cu-\lambda u=f
$$
with leading coefficients $a^{ij}$ which also depend on $x^1$ and $x''$. They are supposed to be measurable with respect to $x^2$, and have small local mean oscillations in the other variables. To be more precise, we impose the following assumption which contains a parameter $\gamma>0$ to be
specified later.
\begin{assumption}                          \label{assum1}
There is a constant $R_0\in (0,1]$ such that the following holds. For any ball $B$ of radius $r\in (0,R_0)$, there exist $\bar a^{ij}=\bar a^{ij}(x^2)$, which depend on the ball $B$ and satisfy \eqref{eq17.53}, such that
$$
\sum_{i,j=1}^d \dashint_{B}|a^{ij}(x)-\bar a^{ij}(x^2)|\,dx \le \gamma.
$$
\end{assumption}
Moreover, we assume that $a^{ij}$ satisfy \eqref{eq17.53}, and $b_i$ and $c$ are measurable functions bounded by a constant $K>0$.

The results stated below are generalization of Theorems \ref{thm1} and \ref{thm2}.

\begin{theorem}[The Dirichlet problem]
                                \label{thm3}
Let $p \in (1,2]$
and $f \in L_p(\bR^d_+)$.
Then there exist constants $\gamma\in (0,1)$ and $N>0$ depending only on $d$, $p$, and $\delta$
such that under Assumption \ref{assum1}
the following hold true.
For any $u \in \WO_p^2(\bR^d_+)$ satisfying
\begin{equation}
                             \label{eq15.16.01}
L u - \lambda u =f\quad \text{in}\,\,\bR^d_+,
\end{equation}
we have
\begin{equation}
                             \label{eq15.16.02}
\lambda\|u\|_{L_{p}(\bR^d_+)}+\lambda^{1/2}
\|Du\|_{L_{p}(\bR^d_+)}+\|D^2u\|_{L_{p}(\bR^d_+)}\le
N\|f\|_{L_{p}(\bR^d_+)},
\end{equation}
provided that $\lambda \ge \lambda_1$,
where $\lambda_1 \ge 0$ is a constant
depending only on $d$, $p$, $\delta$, $K$, and $R_0$. Moreover, for any  $\lambda > \lambda_1$, there exists a unique $u \in W_p^2(\bR^d_+)$ solving \eqref{eq15.16.01} with the Dirichlet boundary condition $u=0$ on $\partial \bR^d_+$.
\end{theorem}

\begin{theorem}[The Neumann problem]
                                \label{thm4}
Let $p \in (2,\infty)$
and $f \in L_p(\bR^d_+)$.
Then there exist constants $\gamma\in (0,1)$ and $N>0$ depending only on $d$, $p$, and $\delta$
such that under Assumption \ref{assum1}
the following hold true.
For any $u \in W_p^2(\bR^d_+)$ satisfying \eqref{eq15.16.01} and $D_1 u=0$ on $x^1=0$, we have \eqref{eq15.16.02}
provided that $\lambda \ge \lambda_1$,
where $\lambda_1 \ge 0$ is a constant
depending only on $d$, $p$, $\delta$, $K$, and $R_0$. Moreover, for any  $\lambda > \lambda_1$, there exists a unique $u \in W_p^2(\bR^d_+)$ solving \eqref{eq15.16.01} with the Neumann boundary condition $D_1 u=0$ on $\partial \bR^d_+$.
\end{theorem}

Our proofs are in the spirit of a perturbation method developed by
Krylov \cite{Krylov_2005} for second-order equations in the whole
space. However, due to the lack of a $C^{2,\alpha}$ interior or boundary regularity of solutions to \eqref{eq8.51}, we are not able to apply the perturbation method directly to the non-divergence equation \eqref{eq15.16.01}. Here we use again the idea of rewriting the equation in a divergence form, and then reduce the problem to the $W^1_p$-estimate for certain divergence form equations.

We give several auxiliary results in the next subsection
and then complete the proofs of Theorems \ref{thm3} and \ref{thm4}
in Subsection \ref{sec4.2}.

\subsection{Estimates of $DD_{x''}u$}
Throughout this and next subsections, we assume that $b\equiv c\equiv 0$.
First we present several estimates for $D_{x''}^2 u$. For
convenience, we set $D_{x''}^2 u\equiv 0$ if $d=2$. The next lemma
is a consequence of the Krylov--Safonov estimate.

\begin{lemma}
                                     \label{lem4.6}
Let $\lambda\ge 0$, $q\in (1,\infty)$, and $r > 0$. Assume that
$a^{ij}$ are independent of $x''$, $u \in C_b^{\infty}(\bar
B_r^+)$, $Lu-\lambda u=0$ in $B_{r}^+$, and either $u$ or $D_1 u$
vanishes on $\Gamma_{r}$. Then there exist
constants $N = N(d, \delta,q)$ and $\alpha=\alpha(d,\delta)\in
(0,1]$ such that
\begin{equation}
                                          \label{eq9.10.1}
 [D^{2}_{x''}u]_{C^\alpha(B_{r/2}^+)}+
 \lambda [u]_{C^\alpha(B_{r/2}^+)}
\le N r^{-\alpha} \left(|D^{2}_{x''}u|^{q}+\lambda^q |u|^q\right)_{B^+_{r}}^{1/q}.
\end{equation}
\end{lemma}
\begin{proof}
For $\lambda=0$, \eqref{eq9.10.1} directly follows from the
Krylov--Safonov estimate since $D^{2}_{x''}u$ satisfies the same
equation. The general case $\lambda>0$ then follows from the same
argument as in the proof of Lemma \ref{lem2.1}.
\end{proof}

Combining Theorem \ref{thm1}, Lemma \ref{lem4.6}, and the corresponding interior estimate gives the following bound of mean oscillations.

\begin{lemma}
                                            \label{lem4.7}
Let $\lambda\ge 0$, $q\in (1,2]$, $r>0$, $\kappa\ge 32$, $x_0\in
\bar\bR^d_+$, and $f\in L_q(B_{\kappa r}^+(x_0))$. Suppose
$a^{ij}=a^{ij}(x^2)$, and $u\in W^2_q(B_{\kappa r}^+(x_0))$
satisfies
\begin{equation*}
                 %                       \label{eq15.21.02}
Lu-\lambda u=f \quad \text{in}\,\,B_{\kappa r}^+(x_0)
\end{equation*}
with the Dirichlet boundary condition $u=0$ on $B_{\kappa
r}(x_0)\cap \partial\bR_+^d$. Then
$$
\dashint_{B^{+}_r(x_0)}\dashint_{B^{+}_r(x_0)}| D^2_{x''} u (x)-
 D^2_{x''} u (y)|^q+\lambda^q| u (x)-
 u (y)|^q \,dx\,dy
$$
\begin{equation}
                                                  \label{osc1}
\leq N\kappa^d \dashint_{B^{+}_{\kappa r}
(x_0)} |f|^q \,dx
 +N\kappa^{-q\alpha} \dashint_{B^{+}_{\kappa r}(x_0)}
 |  D^2_{x''} u|^q +\lambda^q |u|^q\,dx,
\end{equation}
where $\alpha$ is the constant from Lemma \ref{lem4.6}, and the
constant $N$ depends only on $d$, $\delta$, and $q$. The same
estimate holds for $q\in [2,\infty)$ if $D_1u$  vanishes on
$B_{\kappa r}(x_0)\cap \partial\bR_+^d$ instead of $u$.
\end{lemma}
\begin{proof}
We begin with a few reductions. Without loss of generality, we may
assume $x_0=(x_0^1,0,\ldots,0)$ with $x_0^1\ge 0$. By
mollifications, we may also assume $a^{ij}\in C^\infty$. Dilations
show that it suffices to prove the lemma only for $\kappa r=8$. As
before, we consider two cases: i) $x_0^1< 1$ and ii) $x_0^1>1$. We
only prove the first case. The proof of the second case is simpler
(cf. Corollary \ref{cor2.3}).

First we suppose $\lambda>0$. We argue as in the proof of
Corollary \ref{cor2.3}. Since $r=8/\kappa\le 1/4$, we have
\begin{equation}
                                    \label{eq15.22.55}
B^{+}_r(x_0)\subset B^{+}_2 \subset B^{+}_6\subset B^{+}_{\kappa r}(x_0).
\end{equation}
By using a standard density argument, we may assume
$u\in C_b^{\infty}(\bar{B}_{\kappa r}^{+}(x_0))$ and $a^{ij}\in C^\infty(\bR)$. Take a smooth cutoff function $\eta\in C_0^\infty(B_6)$ such that
$$
\eta\equiv 1\quad \text{in}\,\, B_4,\quad 0\le \eta\le 1\quad \text{in}\,\,B_6.
$$
According to Theorem \ref{thm1}, there exists a unique solution $w\in W^2_q(\bR^d_+)$ to the equation
$$
Lw-\lambda w=\eta f\quad \text{in}\,\,\bR^d_+
$$
with the zero Dirichlet boundary condition. Moreover, we have
$$
\|D^2 w\|_{L_q(\bR^d_+)}+\lambda \|w\|_{L_q(\bR^d_+)}
\le N(d,\delta,q)\|\eta f\|_{L_q(\bR^d_+)},
$$
which implies that
\begin{align}
                                                  \label{eq15.22.38}
\dashint_{B^{+}_r(x_0)}|D^2 w|^q+\lambda^q|w|^q \,dx
&\leq N\kappa^d \dashint_{B^{+}_{\kappa r}
(x_0)} |f|^q \,dx,\\
                                        \label{eq15.22.40}
\dashint_{B^{+}_{\kappa r}(x_0)}| D^2 w|^q+\lambda^q|w|^q \,dx
&\leq N \dashint_{B^{+}_{\kappa r}(x_0)} |f|^q \,dx,
\end{align}
By the classical elliptic theory, $w\in C^\infty(\bar\bR^d_+)$.
Now we define
$$
v:=u-v\in C_b^{\infty}(\bar{B}_{\kappa r}^{+}(x_0)),
$$
which clearly satisfies
$$
Lv-\lambda v=0\quad \text{in}\,\,B_4^+
$$
and $v=0$ on $\Gamma_4$. Recall that $r=8/\kappa$.
By Lemma \ref{lem4.6} and \eqref{eq15.22.55},
\begin{align}
                                                  \label{eq15.22.57}
&\dashint_{B^{+}_r(x_0)}\dashint_{B^{+}_r(x_0)}| D^2_{x''} v (x)-
 D^2_{x''} v (y)|^q+\lambda^q| v (x)-
 v (y)|^q \,dx\,dy\nonumber\\
&\,\,\le Nr^{ q\alpha}([D^2_{x''} v]_{C^\alpha(B_2^+)}+\lambda [v]_{C^\alpha(B_2^+)})^q \nonumber\\
&\,\,\le N\kappa^{-q\alpha} \dashint_{B^{+}_{\kappa r}(x_0)}
 |  D^2_{x''} v|^q +\lambda^q |v|^q\,dx.
\end{align}
Since $u=v+w$, combining \eqref{eq15.22.38}, \eqref{eq15.22.40}, \eqref{eq15.22.57}, and the triangle inequality, we immediately get \eqref{osc1}. The case $\lambda=0$ follows from the case $\lambda>0$ by taking the limit as $\lambda\searrow 0$.
%When $\lambda=0$, we take $\lambda_1>0$. Since $Lu-\lambda_1 u=f-\lambda_1 u$ in $B_{\kappa r}^+(x_0)$, by the proof above, we have
%\begin{multline*}
%                %                                  \label{osc10}
%\dashint_{B^{+}_r(x_0)}\dashint_{B^{+}_r(x_0)}| D^2_{x''} u (x)-
% D^2_{x''} u (y)|^q\,dx\,dy\\
%\leq N\kappa^d \dashint_{B^{+}_{\kappa r}
%(x_0)} |f-\lambda_1 u|^q \,dx
% +N\kappa^{-q\alpha} \dashint_{B^{+}_{\kappa r}(x_0)}
% |  D^2_{x''} u|^q +\lambda^q |u|^q\,dx.
%\end{multline*}
%Taking the limit as $\lambda_1\to 0$ yields the desired estimate.

The proof of the last assertion is the same by using Theorem \ref{thm2} in place of Theorem \ref{thm1}.
\end{proof}

Lemma \ref{lem4.7} together with a perturbation argument gives the
next result for general operators $L$ satisfying Assumption
\ref{assum1}.

\begin{lemma}
                                            \label{lem4.8}
Let $\lambda\ge 0$,  $q\in (1,2]$, $\beta\in (1,\infty)$,
$\beta'=\beta/(\beta-1)$, $x_1\in \bar\bR^d_+$, and $f\in
L_{q,\text{loc}}(\bar\bR^d_+)$. Suppose $u\in
W^2_{q,\text{loc}}(\bar\bR^d_+)$ vanishes outside $B^+_{R_0}(x_1)$
and satisfies
\begin{equation*}
                %                        \label{eq15.23.35}
Lu-\lambda u=f \quad \text{in}\,\,\bR^d_+,
\end{equation*}
with the Dirichlet boundary condition $u=0$ on $\partial\bR_+^d$.
Then under Assumption \ref{assum1}, for any $r>0$, $\kappa\ge 32$,
and $x_0\in \bar\bR^d_+$, we have
$$
\dashint_{B^{+}_r(x_0)}\dashint_{B^{+}_r(x_0)}| D^2_{x''} u (x)-
 D^2_{x''} u (y)|^q+\lambda^q| u (x)-
 u (y)|^q \,dx\,dy
$$
$$
\leq N\kappa^d \dashint_{B^{+}_{\kappa r}
(x_0)} |f|^q \,dx
 +N\kappa^{-q\alpha} \dashint_{B^{+}_{\kappa r}(x_0)}
 |  D^2_{x''} u|^q +\lambda^q |u|^q\,dx,
$$
\begin{equation}
                                                  \label{osc2}
+N\kappa^d \Big(\dashint_{B^{+}_{\kappa r}(x_0)}
 |  D^2 u|^{\beta q}\,dx\Big)^{\frac 1 {\beta}}\gamma^{\frac 1 {\beta'}},
\end{equation}
where the constant $N$ depends only on $d$, $\delta$, $\beta$, and
$q$. The same estimate hold for $q\in [2,\infty)$ if $D_1u$
vanishes on $\partial\bR_+^d$ instead of $u$.
\end{lemma}
\begin{proof}
We choose $B=B_{\kappa r}(x_0)$ if $\kappa r<R_0$ and $B=B_{R_0}(x_1)$ if $\kappa r\ge R_0$. For this ball $B$, let $\bar a^{ij}=\bar a^{ij}(x^2)$ be the coefficients given by Assumption \ref{assum1} and $\bar L$ be the elliptic operator with the coefficients $\bar a^{ij}$. Then we have $\bar Lu-\lambda u=\bar f$ in $\bR^d_+$, where $\bar f=f+(\bar a^{ij}-a^{ij})D_{ij}u$. It follows from Lemma \ref{lem4.7} that the left-hand side of \eqref{osc2} is less than
$$
N\kappa^d \dashint_{B^{+}_{\kappa r}
(x_0)} |f+(\bar a^{ij}-a^{ij})D_{ij}u|^q \,dx
 +N\kappa^{-q\alpha} \dashint_{B^{+}_{\kappa r}(x_0)}
 |  D^2_{x''} u|^q +\lambda^q |u|^q\,dx.
$$
Notice that, by H\"older's inequality,
\begin{align*}
&\dashint_{B^{+}_{\kappa r}
(x_0)} |(\bar a^{ij}-a^{ij})D_{ij}u|^q \,dx=\dashint_{B^{+}_{\kappa r}
(x_0)} |1_B(\bar a^{ij}-a^{ij})D_{ij}u|^q \,dx\\
&\,\,\le \Big(\dashint_{B^{+}_{\kappa r}(x_0)}|  D^2 u|^{\beta q}\,dx\Big)^{\frac 1 {\beta}}
\Big(\dashint_{B^{+}_{\kappa r}(x_0)}|1_B(\bar a^{ij}-a^{ij})|^{\beta' q}\,dx\Big)^{\frac 1 {\beta'}}\\
&\,\,\le \Big(\dashint_{B^{+}_{\kappa r}(x_0)}|  D^2 u|^{\beta q}\,dx\Big)^{\frac 1 \beta}
\Big(\dashint_{B}|(\bar a^{ij}-a^{ij})|^{\beta' q}\,dx\Big)^{\frac 1 {\beta'}}\\
&\,\,\le N\Big(\dashint_{B^{+}_{\kappa r}(x_0)}
 |  D^2 u|^{\beta q}\,dx\Big)^{\frac 1 \beta}\gamma^{\frac 1 {\beta'}},
\end{align*}
where the last inequality is due to Assumption \ref{assum1}. Thus
collecting the above inequalities we get \eqref{osc2} immediately.
The last assertion follows from the last assertion of Lemma \ref{lem4.7} by the same proof.
The lemma is proved.
\end{proof}

\begin{corollary}
                                \label{cor4.7}
Let $\lambda\ge 0$, $p\in (1,\infty)$, $x_1\in \bar\bR^d_+$ and
$f\in L_p(\bR^d_+)$. Suppose $u\in W^2_{p}(\bR^d_+)$ vanishes
outside $B^+_{R_0}(x_1)$ and satisfies
\begin{equation*}
                %                        \label{eq15.23.35}
Lu-\lambda u=f \quad \text{in}\,\,\bR^d_+,
\end{equation*}
with the Dirichlet boundary condition $u=0$ on $\partial\bR_+^d$.
Then there exist a constant $\alpha_1=\alpha_1(p)>0$ such that
under Assumption \ref{assum1} the following holds. For any
$\gamma\in (0,1)$, we have
\begin{equation}
                                        \label{eq18.16.59}
\|DD_{x''} u\|_{L_p(\bR^d_+)}+\lambda \|u\|_{L_p(\bR^d_+)} \le
N\gamma^{\alpha_1} \|D^2
u\|_{L_p(\bR^d_+)}+N_1\|f\|_{L_p(\bR^d_+)},
\end{equation}
where $N=N(d,p,\delta)>0$ and $N_1=N_1(d,p,\delta,\gamma)>0$. The
same estimate hold for $p\in (2,\infty)$ if $D_1u$ vanishes on
$\partial\bR_+^d$ instead of $u$.
\end{corollary}
\begin{proof}
We take $q\in (1,2]$ and $\beta\in (1,\infty)$ such that $p>\beta q$. Due to \eqref{eq20.20.20}, \eqref{eq20.20.20b}, and Lemma \ref{lem4.8} above, we obtain a pointwise estimate
\begin{multline}
                                \label{eq18.16.36}
(D_{x''}^2u)^\#(x_0)+\lambda u^\#(x_0)
\le N\kappa^{\frac d q}\big(\cM (|f|^q)\big)^{\frac 1 q}
+N\kappa^{-\alpha}\big(\cM (|D_{x''}^2 u|^q)\big)^{\frac 1 q}\\
+N\kappa^{-\alpha}\lambda \big(\cM (|u|^q)\big)^{\frac 1 q}
+N\kappa^{\frac d q}\gamma^{\frac 1 {\beta' q}}
\big(\cM (|D^2 u|^{\beta q})\big)^{\frac 1 {\beta q}}
\end{multline}
for any $x_0\in \bar\bR^d_+$.
As in the proof of Proposition \ref{prop0}, we deduce from \eqref{eq18.16.36} that
\begin{multline*}
                                  %  \label{eq18.16.38}
\|D_{x''}^2u\|_{L_p(\bR^d_+)}+\lambda \|u\|_{L_p(\bR^d_+)}
\le N\kappa^{\frac d q}\|f\|_{L_p(\bR^d_+)}\\
+N\kappa^{-\alpha}
(\|D_{x''}u\|_{L_p(\bR^d_+)}+\lambda \|u\|_{L_p(\bR^d_+)})
+N\kappa^{\frac d q}\gamma^{\frac 1 {\beta' q}}\|D^2 u\|_{L_p(\bR^d_+)}.
\end{multline*}
By taking $\kappa$ sufficiently large such that $N\kappa^{-\alpha}\le 1/2$, we get
\begin{equation}
                                  \label{eq18.16.38}
\|D_{x''}^2u\|_{L_p(\bR^d_+)}+\lambda \|u\|_{L_p(\bR^d_+)}
\le N\|f\|_{L_p(\bR^d_+)}
+N\gamma^{\frac 1 {\beta' q}}\|D^2 u\|_{L_p(\bR^d_+)}.
\end{equation}
Next observe that for any $\varepsilon>0$
\begin{equation}
                                                    \label{3.57}
\|D_{x'x''}u\|_{L_{p}(\bR^d_+)}\le \varepsilon\|D_{x'}^2 u\|_{L_{p}(\bR^d_+)}
+N(d,p)\varepsilon^{-1}\|D_{x''}^2u\|_{L_{p}(\bR^d_+)},
\end{equation}
which is deduced from
$$
\|D_{x'x''}u\|_{L_{p}(\bR^d_+)}\le N
\|\Delta u\|_{L_{p}(\bR^d_+)}\leq N\|D_{x'}^2u\|_{L_{p}(\bR^d_+)}+N \|D_{x''}^2u\|_{L_{p}(\bR^d_+)}
$$
by scaling in $x'$. Combining \eqref{eq18.16.38} and \eqref{3.57},
we reach \eqref{eq18.16.59} upon choosing
$\varepsilon=\gamma^{1/2}$. The last assertion follows from the
last assertion of Lemma \ref{lem4.8} by using the same proof.
\end{proof}

\subsection{Estimates for divergence form equations}
The following result is a special case of Theorem 2.7 in
\cite{Krylov08}, which can be viewed as a generalization of the
classical Fefferman-Stein theorem.

\begin{theorem}
                         \label{th081201} Let $p
\in (1, \infty)$, and $U,V,F\in L_{1,\text{loc}}(\bR^{d}_+)$.
Assume that we have $|U| \le V$ and, for each $l \in \bZ$ and $C
\in \bC_l$, there exists a measurable function $U^C$ on $C$ such
that $|U| \le U^C \le V$ on $C$ and
\begin{equation*}                            %\label{eq082006}
\int_C |U^C - \left(U^C\right)_C| \,dx \le \int_C F(x) \,dx.
\end{equation*}
Then
$$
\| U \|_{L_p(\bR^{d}_+)}^p \le N(d,p) \|F\|_{L_p(\bR^{d}_+)}\| V
\|_{L_p(\bR^{d}_+)}^{p-1},
$$
provided that $F,V\in L_p(\bR^{d}_+)$.
\end{theorem}

Recall the divergence form operator $\cL$ introduced at the
beginning of Section \ref{sec2}. For any ball $B$ of radius $r\in
(0,R_0)$ and $a^{ij}$ satisfying Assumption \ref{assum1}, we
define
$$
U_B=U_B(x):=|D_1 u|+|D_{x''}u|+|\bar a^{2j}(x^2)D_j
u|+\lambda^{1/2}|u|.
$$

\begin{lemma}
                                        \label{lem4.4}
Let $\lambda>0$, $\beta\in (1,\infty)$, and
$\beta'=\beta/(\beta-1)$ be constants, $x_1\in \bar\bR_+^d$,
$\cL\in \bL_2$, and $g=(g^1,\ldots,g^d),f\in
C_0^\infty(\bar\bR^d_+)$. Suppose that $u\in C_0^\infty(\bar
\bR^d_+)$ vanishing outside $B^+_{R_0}(x_1)$ satisfies
$$
\cL u-\lambda u=\Div g+f\quad \text{in}\,\,\bR^d_+
$$
with the conormal derivative boundary condition $a^{11}D_1 u=g^1$
on $\partial\bR^d_+$. Then under Assumption \ref{assum1}, for any
$r>0$, $\kappa\ge 32$ and $x_0\in \bar\bR^d_+$, we have
\begin{multline}
                                                  \label{eq17.21.22z}
\dashint_{B^{+}_r(x_0)}\dashint_{B^{+}_r(x_0)}|U_B (x)-
U_B (y)|^2\,dx\,dy\\
\leq N\kappa^d \dashint_{B^{+}_{\kappa r} (x_0)}
|g|^2+\lambda^{-1}f^2 \,dx
 +N\kappa^{-2\alpha} \dashint_{B^{+}_{\kappa
 r}(x_0)}\big(U_B\big)^2\,dx\\
 +N\kappa^d \Big(\dashint_{B^{+}_{\kappa r}(x_0)}
 \big(U_B\big)^{2\beta}\,dx\Big)^{1/\beta}\gamma^{1/\beta'},
\end{multline}
where $B=B_{\kappa r}(x_0)$ if $\kappa r<R_0$ and $B=B_{R_0}(x_1)$
if $\kappa r\ge R_0$, $\alpha$ is the constant from Lemma
\ref{lem2.1}, and the constant $N$ depends only on $d$ and
$\delta$. Moreover, for any $l\in \bZ$ and $C\in \bC_l$, there
exists a function $U^C$ on $\bR^d_+$ such that
\begin{equation}
                                        \label{eq19.20.22z}
N^{-1} (|Du|+\lambda^{1/2} |u|)\le U^C\le N(|Du|+\lambda^{1/2}
|u|),
\end{equation}
\begin{equation}                            \label{eq082006}
\int_C |U^C - \left(U^C\right)_C| \,dx \le N\int_C F(x) \,dx,
\end{equation}
where
$$
F=\kappa^{\frac d2}\Big(\cM\big(|g|^2+\lambda^{-1}
f^2\big)\Big)^{\frac 12}+(\kappa^{-\alpha}+\kappa^{\frac
d2}\gamma^{\frac 1 {2\beta'}})
\Big(\cM\big(U^C\big)^{2\beta}\Big)^{\frac 1{2\beta}}.
$$

The same estimates holds for $\cL\in \bL_1$ if the conormal
derivative boundary condition is replaced with the Dirichlet
boundary condition $u=0$ on $\partial\bR_+^d$.
\end{lemma}
\begin{proof}
Inequality \eqref{eq17.21.22z} is a direct consequence of
Corollary \ref{cor2.3} by using a perturbation argument similar to
the proof of Lemma \ref{lem4.8}. Next, for any $C\in \bC_l$, we
find a smallest ball $B(x_0)$ which contains $C$ and define
$U^C=U_B$. Then \eqref{eq19.20.22z} follows from
\eqref{eq19.20.22}. By \eqref{eq17.21.22z} and the triangle
inequality, the left-hand side of \eqref{eq082006} is less than
$NF(x)$ for any $x\in C$, which yields \eqref{eq082006}
immediately.
\end{proof}

The following proposition is an extension of Proposition
\ref{prop0}.

\begin{proposition}
                                \label{prop4.4}
Suppose that either $\cL\in \bL_1$ and $p\in (1,2]$, or 
$\cL\in \bL_2$ and $p\in [2,\infty)$. Let $g=(g^1,\ldots,g^d)$, $f\in
C_0^\infty(\bar\bR^d_+)$. %, and $x_1\in \bar\bR^d_+$. 
Then we can find constants $\gamma\in
(0,1)$ and $N>0$ depending only on $d$, $p$, and $\delta$ such
that under Assumption \ref{assum1} the following holds true. For
any $\lambda\ge \lambda_0$ and $u\in C_{\text{loc}}^\infty(\bar\bR^d_+)$ 
satisfying
$$
\cL u-\lambda u=\Div g+f\quad \text{in}\,\, \bR^d_+
$$
with the conormal derivative boundary condition $a^{1j}D_ju=g^1$
on $\partial\bR^d_+$, we have
\begin{equation}
                \label{eq15.16.35}
\lambda\|u\|_{L_{p}(\bR^d_+)}+\lambda^{1/2}
\|Du\|_{L_{p}(\bR^d_+)}\le N\lambda^{1/2}\|g\|_{L_{p}(\bR^d_+)}
+N\|f\|_{L_{p}(\bR^d_+)}.
\end{equation}
where $\lambda_0\ge 0$ is a constant depending only on $d$,
$\delta$, $p$, and $R_0$.
%Furthermore, for $\lambda=0$ and $f\equiv 0$, we have
%\begin{equation}
%                \label{eq19.46bz}
%\|Du\|_{L_{p}(\bR^d_+)}\le N\|g\|_{L_{p}(\bR^d_+)}.
%\end{equation}
\end{proposition}
\begin{proof}
As in the proof of Proposition \ref{prop0}, we may assume that $a^{ij}\in C^\infty$, $\cL\in \bL_2$, and $p\in (2,\infty)$. First, we consider the case when $u$ vanishes outside $B_{R_0}^+(x_1)$ for some $x_1\in \bar\bR^d_+$.
We define $U$ and $V$ to be respectively the left and
right hand side of \eqref{eq19.20.22z}. Choose $\beta>1$ such that $p>2\beta$.
By Lemma \ref{lem4.4} and Theorem \ref{th081201}, we have
$$
\| U \|_{L_p(\bR^{d}_+)}^p \le N \|F\|_{L_p(\bR^{d}_+)}\| V
\|_{L_p(\bR^{d}_+)}^{p-1},
$$
which yields
$$
\| U \|_{L_p(\bR^{d}_+)} \le N \|F\|_{L_p(\bR^{d}_+)}.
$$
By the definition of $F$ and the maximal function theorem, we bound $\| U \|_{L_p(\bR^{d}_+)}$ by
\begin{align*}
&N \kappa^{\frac d2}\Big\|\cM\big(|g|^2+\lambda^{-1}
f^2\big)\Big\|_{L_{\frac p 2}(\bR^d_+)}^{\frac 12}+N(\kappa^{-\alpha}+\kappa^{\frac
d2}\gamma^{\frac 1 {2\beta'}})
\Big\|\cM\big(U^C\big)^{2\beta}\Big\|^{\frac 1{2\beta}}_{L_{\frac p {2\beta}}(\bR^{d}_+)}\\
&\le N \kappa^{\frac d2}\big\||g|^2+\lambda^{-1}
f^2\big\|_{L_{\frac p 2}(\bR^d_+)}^{\frac 12}+N(\kappa^{-\alpha}+\kappa^{\frac
d2}\gamma^{\frac 1 {2\beta'}})
\big\|\big(U^C\big)^{2\beta}\big\|^{\frac 1{2\beta}}_{L_{\frac p {2\beta}}(\bR^{d}_+)}\\
&\le N \kappa^{\frac d2}\big(\|g\|_{L_p(\bR^{d}_+)}+\lambda^{-\frac 12}
\|f\|_{L_p(\bR^d_+)}\big)+N(\kappa^{-\alpha}+\kappa^{\frac
d2}\gamma^{\frac 1 {2\beta'}})
\|U\|_{L_{p}(\bR^{d}_+)}.
\end{align*}
Upon choosing $\kappa$ sufficiently large and $\gamma$ sufficiently small, we get
$$
\| U \|_{L_p(\bR^{d}_+)} \le N\|g\|_{L_p(\bR^{d}_+)}+N\lambda^{-\frac 12}
\|f\|_{L_p(\bR^d_+)}.
$$
By the definition of $U$, we obtain \eqref{eq15.16.35}
%Finally, \eqref{eq19.46bz} is deduced from \eqref{eq15.16.35} in the same way as \eqref{eq19.46b} is deduced from \eqref{eq19.46}.
%$$
%\lambda\|u\|_{L_{p}(\bR^d_+)}+\sqrt{\lambda}
%\|Du\|_{L_{p}(\bR^d_+)}\le N\lambda^{1/2}\|g\|_{L_{p}(\bR^d_+)}
%+N\|f\|_{L_{p}(\bR^d_+)}
%$$
for any $\lambda> 0$ and $u$ vanishing outside $B_{R_0}^+(x_1)$.
To complete the proof of \eqref{eq15.16.35} for general $u\in C_{\text{loc}}^\infty(\bar\bR^d_+)$, we use a standard partition of unity argument. See, for instance, the proof of \cite[Theorem 5.7]{Krylov_2005}.
The proposition is proved.
\end{proof}
%{\color{red} add a corollary to remove the small support condition...}

Similarly, we have the following estimate for the Dirichlet
problem.
\begin{proposition}
                                \label{lem4.5}
Suppose that $\cL\in \bL_1$ and $p\in [2,\infty)$, or $\cL\in \bL_2$ and $p\in (1,2]$.
Let $g=(g^1,\ldots,g^d)$, $f\in
C_0^\infty(\bar\bR^d_+)$.%, and $x_1\in \bar\bR^d_+$. 
Then we can find constants $\gamma\in
(0,1)$ and $N>0$ depending only on $d$, $p$, and $\delta$ such
that under Assumption \ref{assum1} the following holds true. For
any $\lambda\ge \lambda_0$ and $u\in C_{\text{loc}}^\infty(\bar\bR^d_+)$
satisfying
$$
\cL u-\lambda u=\Div g+f\quad \text{in}\,\, \bR^d_+
$$
with the Dirichlet boundary condition $u=0$ on $\partial\bR^d_+$,
we have \eqref{eq15.16.35}, 
%Furthermore, for $\lambda=0$ and $f\equiv 0$, \eqref{eq19.46bz} holds.
where $\lambda_0\ge 0$ is a constant depending
only on $d$, $\delta$, $p$, and $R_0$.
\end{proposition}
\begin{proof}
The proof is almost the same as that of Proposition \ref{prop4.4}
with obvious modifications.
\end{proof}

\subsection{Proofs of Theorems \ref{thm3} and \ref{thm4}}
                            \label{sec4.2}
We are now ready to prove the main results of this section. The
key step below is to divide both sides by $a^{11}$ so that $D_2 u$
will satisfy a divergence form equation. This idea is in
reminiscence of the classical argument of deriving from the
DeGiorgi--Moser--Nash estimate the $C^{1,\alpha}$ regularity for
non-divergence elliptic equations with measurable coefficients on
the plane; see, for instance, \cite[\S 11.2]{GT}.

\begin{proof}[Proof of Theorem \ref{thm3}]
Owing to mollifications, a density argument, and the method of continuity
it suffices for us to prove the first assertion assuming that the coefficients are
smooth and $u \in C_0^\infty (\bar \bR^d_+)$ vanishes on $x^1=0$. In this case $f \in C_0^\infty (\bar \bR^d_+)$. We shall prove this in two steps.

{\it Step 1.} In this step, we assume that $b\equiv c\equiv 0$, and $u$ vanishes outside $B_{R_0}^+(x_1)$ for some $x_1\in \bar \bR^d_+$. We move $-\lambda u$ and all the second derivatives $DD_{x''}$ to the right-hand side of the equation to get
$$
\sum_{i,j=1}^2 a^{ij}D_{ij}u =f+\lambda u-\sum_{i\, \text{or}\, j>2}a^{ij}D_{ij}u.
$$
Dividing both sides by $a^{11}$ and adding $\Delta_{x''} -\lambda u$ to both sides give
\begin{equation}
                                    \label{eq18.18.00}
\sum_{i,j=1}^2 \tilde a^{ij} D_{ij}u+\Delta_{x''} u-\lambda u=\tilde f\quad \text{in}\,\,\bR^d_+,
\end{equation}
where
$$
\tilde a^{11}=1,\quad \tilde a^{12}=0,\quad \tilde a^{21}=(a^{12}+a^{21})/a^{11},\quad \tilde a^{22}=a^{22}/a^{11},
$$
$$
\tilde f=(a^{11})^{-1}(f+\lambda u-\sum_{i\, \text{or}\, j>2}a^{ij}D_{ij}u)+\Delta_{x''} u-\lambda u.
$$
Thanks to the first assertion of Corollary \ref{cor4.7}, for any
$\gamma\in (0,1)$,
\begin{align}
                                            \label{eq18.18.11}
&\|DD_{x''}u\|_{L_p(\bR^d_+)}+\lambda \|u\|_{L_p(\bR^d_+)}
\le N\gamma^{\alpha_1} \|D^2 u\|_{L_p(\bR^d_+)}+N_1\|f\|_{L_p(\bR^d_+)},\\
                                        \label{eq18.17.53}
&\|\tilde f\|_{L_p(\bR^d_+)} \le N\gamma^{\alpha_1} \|D^2
u\|_{L_p(\bR^d_+)}+N_1\|f\|_{L_p(\bR^d_+)}.
\end{align}
%where $N=N(d,p,\delta)$ and $N_1=N_1(d,p,\delta,\gamma)$.
Here and in the sequel, we denote $N$ to be a constant depending
only on $d$, $p$, $\delta$, and $N_1$ to be a constant depending
on these parameters as well as $\gamma$. Now define a divergence
form operator $\cL$ by
$$
\cL v=\sum_{i,j=1}^2 D_i(\tilde a^{ij} D_{j}v)+\Delta_{x''} v.
$$
Clearly, $\cL\in \bL_2$ and is uniformly nondegenerate with an ellipticity constant depending only on $\delta$. Moreover, $\tilde a^{ij}$ satisfy Assumption \ref{assum1} with $N(\delta)\gamma$ in place of $\gamma$. By differentiating \eqref{eq18.18.00} with respect to $x^2$ and bearing in mind the zero Dirichlet boundary condition, we see that $w:=D_2 u$ satisfies
$$
\cL w-\lambda w=D_2 \tilde f\quad \text{in}\,\,\bR^d_+
$$
with the Dirichlet boundary condition $w=0$ on $x^1=0$. By applying Proposition \ref{lem4.5} to $w$, we get
$$
\|Dw\|_{L_p(\bR^d_+)}\le N(d,p,\delta)\|\tilde f\|_{L_p(\bR^d_+)}
$$
provided that $\gamma<\gamma_1(d,p,\delta)$ and $\lambda\ge \lambda_0(d,p,\delta,R_0)$. This together with
\eqref{eq18.17.53} yields
\begin{equation}
                                    \label{eq18.18.13}
\|DD_2 u\|_{L_p(\bR^d_+)}\le N\gamma^{\alpha_1} \|D^2
u\|_{L_p(\bR^d_+)}+N_1\|f\|_{L_p(\bR^d_+)}
\end{equation}
Now the only missing term $D_1^2u$ can be estimated by combining
\eqref{eq18.18.11}, \eqref{eq18.18.13}  and using the equation
itself. Therefore, we deduce
\begin{equation*}
                                    %\label{eq18.18.28}
\|D^2u\|_{L_p(\bR^d_+)}+\lambda \|u\|_{L_p(\bR^d_+)}\le
N\gamma^{\alpha_1} \|D^2
u\|_{L_p(\bR^d_+)}+N_1\|f\|_{L_p(\bR^d_+)}
\end{equation*}
if $\gamma<\gamma_1$. Upon choosing $\gamma$ even smaller and
using the interpolation inequality, we obtain \eqref{eq15.16.02}.

{\it Step 2.} We now remove the additional assumptions in the previous step. First, the assumption that $u$ vanishes outside $B_{R_0}^+$ can be dropped by using a partition of unity as in Proposition \ref{prop4.4}. For nonzero $b$ and $c$, we move all the lower order terms in $Lu$ to the right-hand side of the equation:
$$
a^{ij}D_{ij}u-\lambda u=f-b_iD_iu-cu
$$
By the estimate proved above and the boundedness of $b,c$, we have
\begin{multline*}
\lambda\|u\|_{L_{p}(\bR^d_+)}+\lambda^{1/2}\|Du\|_{L_{p}(\bR^d_+)}+\|D^2u\|_{L_{p}(\bR^d_+)}\\
\le N\|f\|_{L_{p}(\bR^d_+)}+NK\|u\|_{L_{p}(\bR^d_+)}+NK\|Du\|_{L_{p}(\bR^d_+)}.
\end{multline*}
Bearing in mind that $N$ is independent of $\lambda$, we take $\lambda_1=\lambda_1(d,p,\delta,R_0,K)(\ge \lambda_0)$ sufficiently large such that for $\lambda\ge\lambda_1$, the second and the third terms on the right-hand side above can be absorbed to the left-hand side. Thus, we get \eqref{eq15.16.02}. The theorem is proved.
\end{proof}

\begin{proof}[Proof of Theorem \ref{thm4}]
The proof is similar to that of Theorem \ref{thm3}. Therefore, we only point out the differences. The function $w$ still satisfies the same equation, but instead of the Dirichlet boundary condition, it satisfies the conormal boundary condition $D_1 w=0$ on $x^1=0$. So we apply Proposition \ref{prop4.4} instead of Proposition \ref{lem4.5}, and use the second assertion of Corollary \ref{cor4.7} instead of the first one. The remaining proof is the same.
\end{proof}

\mysection{Applications}
                                        \label{sec5}
In this section, we discuss several applications of our results in Sections \ref{sec3} and \ref{sec4}. Consider the elliptic equation
\begin{equation}
                                            \label{eq5.1}
Lu-\lambda u:=a^{ij}D_{ij}u+b_iD_iu+cu-\lambda u=f
\end{equation}
with the zero Dirichlet boundary condition in a convex wedge $\Omega_\theta=\cO_\theta\times \bR^{d-2}$, where $\theta\in (0,\pi)$ and
$$
\cO_\theta=\{x'\in\bR^2\,|\,x_1> |x'|\cos(\theta/2)\}.
$$
As in Section \ref{sec4}, we suppose that $a^{ij}$ satisfy the ellipticity condition \eqref{eq17.53}, and $b_i$ and $c$ are measurable functions bounded by a constant $K>0$. Furthermore, we assume that $a^{ij}$ have small local mean oscillations with respect to $x$, i.e., there is a constant $R_0\in (0,1]$ such that the following holds with a parameter $\gamma>0$ to be chosen.

\begin{assumption}                          \label{assum1b}
For any ball $B$ of radius $r\in (0,R_0)$,
$$
\sum_{i,j=1}^d \dashint_{B}|a^{ij}(x)-(a^{ij})_B|\,dx \le \gamma,\quad\text{where}\,\,
(a^{ij})_B=\dashint_{B}a^{ij}(x)\,dx.
$$
\end{assumption}
We have the following solvability result for \eqref{eq5.1} in the convex wedge $\Omega_\theta$.
\begin{theorem}
                                    \label{thm5.1}
Let $p \in (1,2]$
and $f \in L_p(\Omega_\theta)$.
Then there exist constants $\gamma\in (0,1)$ and $N>0$ depending only on $d$, $p$, $\delta$, and $\theta$ such that under Assumption \ref{assum1b}
the following hold true.
For any $u \in \WO_p^2(\Omega_\theta)$ satisfying
\begin{equation}
                             \label{eq15.16.01z}
L u - \lambda u =f\quad \text{in}\,\,\Omega_\theta,
\end{equation}
we have
\begin{equation*}
\lambda\|u\|_{L_{p}(\Omega_\theta)}+\lambda^{1/2}
\|Du\|_{L_{p}(\Omega_\theta)}+\|D^2u\|_{L_{p}(\Omega_\theta)} \le
N\|f\|_{L_{p}(\Omega_\theta)},
\end{equation*}
provided that $\lambda \ge \lambda_0$,
where $\lambda_0 \ge 0$ is a constant
depending only on $d$, $p$, $\delta$, $K$, $R_0$, and $\theta$. Moreover,
for any  $\lambda > \lambda_0$, there exists a unique $u \in W_p^2(\Omega_\theta)$ solving \eqref{eq15.16.01z} with the Dirichlet boundary condition $u=0$ on $\partial \Omega_\theta$. Finally, if $b=c=0$ and $a^{ij}$ are constants, then we can take $\lambda_0=0$.
\end{theorem}

For the proof, first by a linear transformation of the coordinates
one may assume that $\theta=\pi/2$ and $\cO_{\pi/2}$ is the first quadrant $\{x'\in\bR^2\,|\,x^1>0,x^2>0\}$. Now take odd/even extensions of the equation with respect to $x^2$ as follows:
$$
\tilde a^{ij}(x)={\text{sgn}(x^2)}a^{ij}(x^1,|x^2|,x'')\quad \text{for}\,\,i=2,j\neq 2\,\,\text{or}\,\,j=2,i\neq 2,
$$
$$
\tilde a^{ij}(x)=a^{ij}(x^1,|x^2|,x'')\quad \text{otherwise},
$$
and
$$
\tilde b^2(x)={\text{sgn}(x^2)}b^2(x^1,|x^2|,x''),\quad \tilde b^j(x)=b^j(x^1,|x^2|,x''),\,\,j\neq 2,
$$
$$
\tilde c(x)=c(x^1,|x^2|,x''),\quad \tilde f(x)= {\text{sgn}(x^2)} f(x^1,|x^2|,x''),
$$
$$
\tilde u(x)= {\text{sgn}(x^2)} u(x^1,|x^2|,x'').
$$
It is easily seen that the new coefficients $\tilde a^{ij}$ satisfy Assumption \ref{assum1} with $2\gamma$. Moreover, we have $\tilde f\in L_{p}(\bR^d_+)$. Let $\tilde L$ be the elliptic operator with coefficients $\tilde a^{ij},\tilde a^i,\tilde {b}^i,\tilde c$. Then $u$ satisfies \eqref{eq15.16.01z} in $\Omega_{\pi/2}$ with the zero Dirichlet boundary condition if and only if $\tilde u$ satisfies
$$
\tilde L\tilde u-\lambda \tilde u=\tilde f\quad\text{in}\,\,\bR^d_+
$$
with the same zero Dirichlet boundary condition. Therefore, the first two assertions of Theorem \ref{thm5.1} follow immediately from Theorem \ref{thm3}. The last assertion is a consequence of Theorem \ref{thm1}. We note that by Remark \ref{rem3.2}, Theorem \ref{thm5.1} actually holds for $p\in (1,2+\varepsilon)$, where $\varepsilon>0$ is a constant which depends only on $d$, $\delta$, and $\theta$.

As a further application, under the same conditions of coefficients, we also obtain the $L_p$-solvability of \eqref{eq5.1} in $\Omega=\cO\times\bR^{d-2}$ with the zero Dirichlet boundary condition, where $\cO$ is a convex polygon in $\bR^2$. Indeed, this is deduced from Theorem \ref{thm5.1} by using a  partition of unity argument. We omit the details.

Another application of Theorem \ref{thm3} is the $W^2_p$-solvability of the Dirichlet
problem for the equation
$$
a^{ij}D_{ij} u=f
$$
in the unit ball $B_1$ when $p\in (1,2]$. Here
we assume that
$a^{ij}$ are piecewise VMO in the upper and lower half balls $B_1^+$ and $B_1^-$. A particular case is that $a^{ij}$ are piecewise constants in $B_1^+$ and $B_1^-$. This
equation can be solved by following the steps in Chapter 11 of
\cite{Kr08}. We notice that when locally flattening the boundary, one gets an equation with leading coefficients which are either VMO or satisfy Assumption \ref{assum1}.

%==================================================

\section*{Acknowledgement}
The author is grateful to Nicolai V. Krylov and Doyoon Kim for their helpful comments.

\appendix

\mysection{Proof of Lemma \ref{lem0}} Denote $\beta=1-2/p_0$. By
the triangle inequality, we have
$$
\sup_{\substack{x,y \in \Omega\\x\ne y}}\frac{|f(x) - f(y)|}{|x-y|^{\beta}}
\le
\sup_{\substack{x',y'\in \tilde B_1^+,x''\in \hat B_1\\x'\ne y'}}\frac{|f(x',x'') - f(y',x'')|}{|x'-y'|^{\beta}}
$$
$$
+ \sup_{\substack{x'\in \tilde B_1^+,x'',y''\in \hat B_1\\x''\ne y''}}
\frac{|f(x',x'') - f(x',y'')|}{|x''-y''|^{\beta}}
:= I_1 + I_2.
$$

{\em Estimate of $I_1$:} By the Sobolev embedding theorem, for any fixed $x''\in \hat B_1$, $f(x',x'')$ as a function of $x'\in \tilde B_1^+$ satisfies
\begin{equation}                            \label{eq01}
%\sup_{\substack{x',y'\in \tilde B_1^+,x''\in \hat B_1\\x'\ne y'}}
%\frac{|f(x',x'')-f(y',x'')|}{|x'-y'|^{\beta}}
I_1\le N \| f(\cdot,x'') \|_{W_{p_0}^1(\tilde B_1^+)}.
\end{equation}
Recall that $k\ge d/2$ and $p_0>2$. By the Sobolev embedding theorem again, for any fixed $x'\in \tilde B_1^+$, $f(x',x'')$ and $D_{x'}f(x',x'')$ as functions of $x'' \in \hat B_1$ satisfy
\begin{multline*}
\sup_{x'' \in \hat B_1}\left(|f(x',x'')|+|D_{x'}f(x',x'')|\right)\\
\le N\|f(x',\cdot)\|_{W_{p_0}^k(\hat B_1)}+
N\|D_{x'}f(x',\cdot)\|_{W_{p_0}^k(\hat B_1)}.
\end{multline*}
This implies that, for any $x'' \in \hat B_1$,
\begin{multline*}
\int_{\tilde B_1^+}|f(x',x'')|^{p_0}\, dx'
+ \int_{\tilde B_1^+}|D_{x'}f(x',x'')|^{p_0}\,dx'\\
\le N \sum_{i \le 1}\sum_{j\le k} \|
D^i_{x'}D^{j}_{x''}f\|^{p_0}_{L_{p_0}(\Omega)}.
\end{multline*}
This combined with \eqref{eq01} shows that
\begin{equation}
                                    \label{eq27.14.23}
%\sup_{\substack{x_1, y_1\in(-1,1), x_1 \ne y_1\\(t,x') \in (0,1)\times B'_1}}
%\frac{|U(t,x_1,x')-U(t,y_1,x')|}{|x_1-y_1|^{1/2}}
I_1 \le N \sum_{i \le 1}\sum_{j\le k} \|
D^i_{x'}D^{j}_{x''}f\|_{L_{p_0}(\Omega)}.
\end{equation}

{\em Estimate of $I_2$:} Again
using the Sobolev embedding theorem, for each $x'\in \tilde B_1^+$,
$f(x',x'')$ as a function of $x'' \in \hat B_1$ satisfies
\begin{equation}                            \label{eq02}
%\sup_{\substack{(t,x'), (s,y') \in Q_1'\\(t,x')\ne (s,y')}}
%\frac{|U(t,y_1,x')-U(s,y_1,y')|}{|t-s|^{1/4}+|x'-y'|^{1/2}}
I_2\le N \| f(x', \cdot) \|_{W_{p_0}^k(\hat B_1)}.
\end{equation}
For each $j \le k$ and $x''\in \hat B_1$,
$D^j_{x''}f(x',x'')$ as a function of $x' \in \tilde B_1^+$ satisfies
$$
\sup_{x'\in \tilde B_1^+}|D^j_{x''}f(x',x'')|
\le N\|D^j_{x''}f(\cdot,x'')\|_{W^1_{p_0}(\tilde B_1^+)}.
$$
This together with \eqref{eq02} gives
\begin{equation}
                                        \label{eq27.14.24}
%\sup_{\substack{y_1\in(0,1)\\(t,x'), (s,y') \in (0,1)\times B'_1\\(t,x')\ne (s,y')}}
%\frac{|U(t,y_1,x')-U(s,y_1,y')|}{|t-s|^{\gamma}+|x'-y'|^{\gamma}}
I_2 \le N\sum_{i \le 1}\sum_{ j\le k}
\|D^i_{x'}D^{j}_{x''}f\|_{L_{p_0}(\Omega)}.
\end{equation}
Combining \eqref{eq27.14.23} and \eqref{eq27.14.24} completes the proof of the lemma.

\end{document}